\documentclass[a4paper,12pt,intlimits,oneside]{amsart}

\usepackage{enumerate}
\usepackage{amsfonts,amsmath}

\usepackage{latexsym,amssymb}

\newcommand{\comment}[1]{}

\newcommand{\bR}{{\mathbb R}}

\newcommand{\bZ}{{\mathbb Z}}

\def\C{{\mathcal C}}

\def\BMO{B\! M\! O}

\newcounter{rea}
\newcounter{rek}
\newcounter{res}
\begin{document}

\title[Bilinear decompositions and  commutators]{Bilinear decompositions and  commutators of singular integral operators}         

\author{Luong Dang  KY}    
\address{MAPMO-UMR 6628,
D\'epartement de Math\'ematiques, Universit\'e d'Orleans, 45067
Orl\'eans Cedex 2, France} 
\email{{\tt dangky@math.cnrs.fr}}


\begin{abstract}

Let $b$ be a $BMO$-function. It is well-known that the linear commutator $[b, T]$ of a Calder\'on-Zygmund operator $T$  does not, in general, map continuously $H^1(\mathbb R^n)$ into $L^1(\mathbb R^n)$. However, P\'erez showed that if $H^1(\mathbb R^n)$ is replaced by a suitable atomic subspace $\mathcal H^1_b(\mathbb R^n)$ then the commutator is continuous from  $\mathcal H^1_b(\mathbb R^n)$ into $L^1(\mathbb R^n)$. In this paper, we find the largest subspace $H^1_b(\mathbb R^n)$ such that all commutators of Calder\'on-Zygmund operators are continuous from  $H^1_b(\mathbb R^n)$ into $L^1(\mathbb R^n)$. Some equivalent characterizations of  $H^1_b(\mathbb R^n)$ are also given. We also study the commutators $[b,T]$ for $T$ in a class $\mathcal K$ of sublinear operators containing almost all important operators in harmonic analysis. When $T$ is linear, we prove that there exists a bilinear operators $\mathfrak R= \mathfrak R_T$ mapping continuously $H^1(\mathbb R^n)\times BMO(\mathbb R^n)$ into $L^1(\mathbb R^n)$ such that for all $(f,b)\in H^1(\mathbb R^n)\times BMO(\mathbb R^n)$, we have
\begin{equation}\label{abstract 1}
[b,T](f)= \mathfrak R(f,b) + T(\mathfrak S(f,b)),
\end{equation}
where $\mathfrak S$ is a bounded bilinear operator from $H^1(\mathbb R^n)\times BMO(\mathbb R^n)$ into $L^1(\mathbb R^n)$  which does not depend on $T$. In the particular case of $T$  a Calder\'on-Zygmund operator satisfying $T1=T^*1=0$ and $b$  in $BMO^{\rm log}(\mathbb R^n)$-- the  generalized $\BMO$ type space  that has been introduced by Nakai and Yabuta to characterize  multipliers of $\BMO(\bR^n)$ --we prove that the commutator $[b,T]$ maps continuously $H^1_b(\mathbb R^n)$ into $h^1(\mathbb R^n)$. Also, if $b$ is in $BMO(\mathbb R^n)$ and $T^*1 = T^*b = 0$, then the commutator $[b, T]$  maps continuously $H^1_b (\mathbb R^n)$ into $H^1(\mathbb R^n)$. When $T$ is sublinear, we prove that  there exists a bounded subbilinear operator $\mathfrak R= \mathfrak R_T: H^1(\mathbb R^n)\times BMO(\mathbb R^n)\to L^1(\mathbb R^n)$  such that for all $(f,b)\in H^1(\mathbb R^n)\times BMO(\mathbb R^n)$, we have 
\begin{equation}\label{abstract 2}
|T(\mathfrak S(f,b))|- \mathfrak R(f,b)\leq |[b,T](f)|\leq \mathfrak R(f,b) + |T(\mathfrak S(f,b))|.
\end{equation}

The bilinear decomposition (\ref{abstract 1})  and the subbilinear decomposition (\ref{abstract 2})  allow us to give a general overview of all known weak and strong $L^1$-estimates.
\end{abstract}

\maketitle
\newtheorem{theorem}{Theorem}[section]
\newtheorem{lemma}{Lemma}[section]
\newtheorem{proposition}{Proposition}[section]
\newtheorem{remark}{Remark}[section]
\newtheorem{corollary}{Corollary}[section]
\newtheorem{definition}{Definition}[section]
\newtheorem{example}{Example}[section]
\numberwithin{equation}{section}
\newtheorem{Theorem}{Theorem}[section]
\newtheorem{Lemma}{Lemma}[section]
\newtheorem{Proposition}{Proposition}[section]
\newtheorem{Remark}{Remark}[section]
\newtheorem{Corollary}{Corollary}[section]
\newtheorem{Definition}{Definition}[section]
\newtheorem{Example}{Example}[section]

\section{Introduction}
Given a function $b$ locally integrable on $\mathbb R^n$, and a Calder\'on-Zygmund operator $T$, we consider the linear commutator $[b, T]$ defined for smooth, compactly supported functions $f$ by
$$[b, T](f)=bT(f) - T(bf).$$

A classical result of R. Coifman, R. Rochberg and G. Weiss (see \cite{CRW}), states that the commutator $[b,T]$ is continuous on $L^p(\mathbb R^n)$ for $1 <p<\infty$, when $b\in BMO(\mathbb R^n)$. Unlike the theory of Calder\'on-Zygmund operators, the proof of this result does not rely on a weak type $(1, 1)$ estimate for $[b, T]$. In fact, it was shown in \cite{Pe} that, in general, the linear commutator fails to be of weak type $(1, 1)$, when $b$ is in $BMO(\mathbb R^n)$. Instead, an endpoint theory was provided for this operator. It is well-known that any singular integral operator maps $H^1(\mathbb R^n)$ into $L^1(\mathbb R^n)$. However, it was observed in \cite{HST} that the commutator $[b, H]$ with $b$ in $BMO(\mathbb R)$, where $H$ is Hilbert transform on $\mathbb R$, does not map, in general, $H^1(\mathbb R)$ into $L^1(\mathbb R)$.  Instead of this, the weak type estimate $(H^1, L^1)$ for $[b,T]$ is well-known, see for example \cite{LiuL, LWY, YHL}.  Remark that intuitively one would like to write
$$[b,T](f)= \sum_{j=1}^\infty \lambda_j (b-b_{B_j})T(a_j)- T\Big(\sum_{j=1}^\infty \lambda_j (b-b_{B_j})a_j\Big),$$
where $f=\sum_{j=1}^\infty \lambda_j a_j $ a atomic decomposition of $f$ and $b_{B_j}$ the average of $b$ on $B_j$. This is equivalent to ask for a commutation property
\begin{equation}\label{equality}
\sum_{j=1}^\infty \lambda_j b_{B_j}T(a_j)= T\Big(\sum_{j=1}^\infty \lambda_j b_{B_j} a_j\Big).
\end{equation}

Even if most authors, for instance in \cite{LiuL, LWY, YHL, ZH, LLM, WL, LinL}, implicitely use (\ref{equality}), one must be careful at this point. Indeed, the equality (\ref{equality}) is not clear since the two series $\sum_{j=1}^\infty \lambda_j b_{B_j}T(a_j)$ and $\sum_{j=1}^\infty \lambda_j b_{B_j} a_j$  are not yet  well-defined, in general. We refer the reader to \cite{Bo}, Section 3,  to be convinced that one must be careful with Equality (\ref{equality}).

Although the commutator $[b,T]$ does not map continuously, in general,  $H^1(\mathbb R^n)$ into $L^1(\mathbb R^n)$, following P\'erez \cite{Pe} one can find a subspace $\mathcal H^1_b(\mathbb R^n)$ of $H^1(\mathbb R^n)$ such that $[b,T]$ maps continuously  $\mathcal H^1_b(\mathbb R^n)$ into $L^1(\mathbb R^n)$. Recall that (see \cite{Pe}) a function $a$ is a $b$-atom if 

i) supp $a\subset Q$ for some cube $Q$,

ii) $\|a\|_{L^\infty}\leq |Q|^{-1}$,

iii) $\int_{\mathbb R^n} a(x)dx=\int_{\mathbb R^n} a(x)b(x)dx= 0$.\\
The space $\mathcal H^1_b(\mathbb R^n)$ consists of the subspace of $L^1(\mathbb R^n)$ of functions $f$ which can be written as $f=\sum_{j=1}^\infty \lambda_j a_j$ where $a_j$ are $b$-atoms, and $\lambda_j$ are complex numbers with $\sum_{j=1}^\infty |\lambda_j|<\infty$.

In \cite{Pe} the author showed that the commutator $[b,T]$ is bounded from $\mathcal H^1_b(\mathbb R^n)$ into $L^1(\mathbb R^n)$ by establishing that 
\begin{equation}\label{Perez 1}
\sup\{\|[b,T](a)\|_{L^1}: a \,\mbox{is a}\; b{\rm -atom}\}<\infty.
\end{equation}
This leaves a gap in the proof which we fill here (see below). Indeed, as it is pointed out in \cite{Bo}, there exists a linear operator $U$ defined on the space of all finite linear combination of $(1,\infty)$-atoms satisfying
$$\sup\{\|U(a)\|_{L^1}: a \,\mbox{is a}\, (1,\infty){\rm -atom}\}<\infty,$$
 but which does not admit an extension to a bounded operator from $H^1(\mathbb R^n)$ into $L^1(\mathbb R^n)$. From this result, we see that Inequality (\ref{Perez 1}) does not suffice to conclude that $[b,T]$ is bounded from $\mathcal H^1_b(\mathbb R^n)$ into $L^1(\mathbb R^n)$. In the setting of $H^1(\mathbb R^n)$, it is well-known (see \cite{MSV1} or \cite{YZ} for details) that  a linear operator $U$ can be extended to a  bounded operator from $H^1(\mathbb R^n)$ into $L^1(\mathbb R^n)$ if for some $1<q<\infty$, we have
$$\sup\{\|U(a)\|_{L^1}: a \,\mbox{is a}\, (1,q){\rm -atom}\}<\infty.$$
It follows from the fact that the finite atomic norm on $H^{1,q}_{\rm fin}(\mathbb R^n)$ is equivalent to the standard infinite atomic decomposition norm on $H^{1,q}_{\rm ato}(\mathbb R^n)$ through the grand maximal function characterization of $H^1(\mathbb R^n)$. However, one can not use this method in the context of $\mathcal H^1_b(\mathbb R^n)$.

Also, a natural question arises: can one  find {\sl the largest subspace of $H^1(\mathbb R^n)$} (of course, this space contains $\mathcal H^1_b(\mathbb R^n)$, see also Theorem \ref{compare}) such that all commutators $[b,T]$ of Calder\'on-Zygmund operators are bounded from this space into $L^1(\mathbb R^n)$? For $b\in BMO(\mathbb R^n)$, a  non-constant  function, we consider the space $H^1_b(\mathbb R^n)$ consisting of all $f\in H^1(\mathbb R^n)$ such that the (sublinear) commutator $[b,\mathfrak M]$ of $f$ belongs to $L^1(\mathbb R^n)$ where $\mathfrak M$ is the nontangential grand maximal operator (see Section 2). The norm on $H^1_b(\mathbb R^n)$ is defined by $\|f\|_{H^1_b}:= \|f\|_{H^1}\|b\|_{BMO}+ \|[b,\mathfrak M](f)\|_{L^1}$. Here we just consider $b$ is a non-constant $BMO$-function since the commutator $[b,T]=0$ if $b$ is a constant function. Then, we  prove that  $[b,T]$ is bounded from $H^1_b(\mathbb R^n)$ into $L^1(\mathbb R^n)$ for every Calder\'on-Zygmund singular integral operator $T$ (in fact it holds for all $T\in \mathcal K$, see below). Furthermore, $H^1_b(\mathbb R^n)$ {\sl is the largest space having this property} (see Remark \ref{the largest space}). This answers the question above. Besides, we also consider the class $\mathcal K$  of all sublinear operators $T$, bounded from $H^1(\mathbb R^n)$ into $L^1(\mathbb R^n)$, satisfying the condition
$$\|(b-b_Q)Ta\|_{L^1}\leq C \|b\|_{BMO}$$
for all $BMO$-function $b$, $H^1$-atom $a$  related to the cube $Q$. Here $b_Q$ denotes the average of $b$ on $Q$, and $C>0$ is a constant independent of $b,a$. This class $\mathcal K$ contains almost all important operators in harmonic analysis: Calder\'on-Zygmund type operators, strongly singular integral operators, multiplier operators,  pseudo-differential operators, maximal type operators, the area integral operator of Lusin, Littlewood-Paley type operators, Marcinkiewicz operators,  maximal Bochner-Riesz operators, etc... (See Section 4). When $T$ is linear and belongs to $\mathcal K$, we prove that there exists a bounded bilinear operators $\mathfrak R= \mathfrak R_T: H^1(\mathbb R^n)\times BMO(\mathbb R^n)\to L^1(\mathbb R^n)$ such that for all $(f,b)\in H^1(\mathbb R^n)\times BMO(\mathbb R^n)$, we have the following bilinear decomposition
\begin{equation}\label{Bilinear decomposition}
[b,T](f)= \mathfrak R(f,b) + T(\mathfrak S(f,b)),
\end{equation}
where $\mathfrak S$ is a bounded bilinear operator from $H^1(\mathbb R^n)\times BMO(\mathbb R^n)$ into $L^1(\mathbb R^n)$  which does not depend on $T$ (see Section 3). This bilinear decomposition is strongly related to our previous result in \cite{BGK} on paraproduct and product on $H^1(\mathbb R^n)\times BMO(\mathbb R^n)$.

 We then prove that $[b,T]$ is bounded from $H^{1}_b(\mathbb R^n)$ into $L^1(\mathbb R^n)$ (see Theorem \ref{new2}) via  Bilinear decomposition (\ref{Bilinear decomposition}) (see Theorem \ref{Th1}) and some characterizations of $H^{1}_b(\mathbb R^n)$ (see Theorem \ref{Bonami}). Furthermore, by using  the weak convergence theorem in $H^1(\mathbb R^n)$ of Jones and Journ\'e, we prove that $\mathcal H^{1}_b(\mathbb R^n)\subset H^{1}_b(\mathbb R^n)$  (see Theorem \ref{compare}). These allow us to fill the gap  mentioned above in \cite{Pe}.

 On the other hand, as an  immediate corollary of Bilinear decomposition (\ref{Bilinear decomposition}), we also obtain the weak type estimate $(H^1, L^1)$ for the commutator $[b,T]$, where $T$ is a Calder\'on-Zygmund type operator, a strongly singular integral operator, a multiplier operator or a  pseudo-differential operator. We  also  point out  that weak type estimates and Hardy type estimates for  the (linear) commutators of multiplier operators and of strongly singular Calder\'on-Zygmund operators have been studied recently (see  \cite{ZH, LLM, WL} for the multiplier operators and  \cite{LinL} for strongly singular Calder\'on-Zygmund operators).

Next, two natural questions for Hardy-type estimates of the commutator $[b, T]$ arised: when does $[b, T]$ map $H^1_
b (\mathbb R^n)$ into $h^1(\mathbb R^n)$ and when does $[b, T]$ map $H^1_b (\mathbb R^n)$ into $H^1(\mathbb R^n)$? 

This paper gives two sufficient conditions for the above two questions. Recall that  $\BMO^{\rm log}(\bR^n)$ --the  generalized $\BMO$ type space  that has been introduced by Nakai and Yabuta  \cite {NY} to characterize  multipliers of $\BMO(\bR^n)$-- is the space of all locally integrable functions $f$ such that
$$\|f\|_{BMO^{\rm log}}:= \sup\limits_{B(a,r)}\frac{|\log r|+ \log (e+|a|)}{|B(a,r)|}\int_{B(a,r)} |f(x)- f_{B(a,r)}|dx<\infty.$$

We obtain that if $T$ is a Calder\'on-Zygmund operator satisfying $T1 = T^*1 = 0$ and $b$ is in $BMO^{\rm log}(\mathbb R^n)$, then the linear commutator $[b, T]$ maps continuously $H^1_b (\mathbb R^n)$ into $h^1(\mathbb R^n)$. This gives a sufficient condition to the first problem. For the second one, we prove that if $T$ is a Calder\'on-Zygmund operator satisfying $T^*1 = T^*b = 0$ and $b$ is in $BMO(\mathbb R^n)$, then the
linear commutator $[b, T]$ maps continuously $H^1_b (\mathbb R^n)$ into $H^1(\mathbb R^n)$.

 A difficult point to prove the first result is that we have to deal directly  with $f\in H^{1}_b(\mathbb R^n)$. It would be easier to do it for atomic type Hardy spaces as in the case of $\mathcal H^1_b(\mathbb R^n)$. However, we do not know  {\sl  whether there exists an atomic characterization for the space $ H^{1}_b(\mathbb R^n)$. This is still an open question}.

Let $X$ be a Banach space. We say that an operator $T: X\to L^1(\mathbb R^n)$ is a sublinear operator if for all $f,g\in X$ and $\alpha,\beta\in \mathbb C$, we have
$$|T(\alpha f+\beta g)(x)|\leq |\alpha||Tf(x)|+ |\beta||Tg(x)|.$$
Obviously, a linear operator $T: X\to L^1(\mathbb R^n)$ is a sublinear operator. We also say that a operator $\mathfrak T:H^1(\mathbb R^n)\times BMO(\mathbb R^n)\to L^1(\mathbb R^n)$ is a subbilinear operator if for all $(f,g)\in H^1(\mathbb R^n)\times BMO(\mathbb R^n)$ the operators $\mathfrak T(f,\cdot): BMO(\mathbb R^n)\to L^1(\mathbb R^n)$ and $\mathfrak T(\cdot,g): H^1(\mathbb R^n)\to L^1(\mathbb R^n)$ are sublinear operators.

In this paper, we also obtain the subbilinear decomposition for sublinear commutator. More precisely, when $T\in \mathcal K$ is a sublinear operator, we prove that  there exists a bounded subbilinear operator $\mathfrak R= \mathfrak R_T: H^1(\mathbb R^n)\times BMO(\mathbb R^n)\to L^1(\mathbb R^n)$ so that for all $(f,b)\in H^1(\mathbb R^n)\times BMO(\mathbb R^n)$, we have
\begin{equation}\label{subbilinear decomposition}
|T(\mathfrak S(f,b))| - \mathfrak R(f,b)\leq |[b,T](f)|\leq \mathfrak R(f,b) + |T(\mathfrak S(f,b))|.
\end{equation}
Then, the strong type estimate $(H^{1}_b, L^1)$ and the weak type estimate $(H^1,L^1)$ of the commutators of Littlewood-Paley type operators, of  Marcinkiewicz operators,  and of maximal Bochner-Riesz operators, can be viewed as an immediate corollary of (\ref{subbilinear decomposition}). When $H^{1}_b(\mathbb R^n)$ is replaced by  $\mathcal H^{1}_b(\mathbb R^n)$,  these type of estimates have also been considered recently (see for example \cite{Liu, CD, LX, LLL, LT, LL}).

Let us emphasize the three main purposes of this paper. First, we want to give the bilinear (resp., subbilinear) decompositions for the linear (resp., sublinear) commutators. Second, we find the largest subspace of $H^1(\mathbb R^n)$  such that all commutators of Calder\'on-Zygmund operators map continuously this space into $L^1(\mathbb R^n)$. Finally, we obtain the $(H^1_b, h^1)$ and $(H^1_b, H^1)$ type estimates for commutators of Calder\'on-Zygmund operators.

Our paper is organized as follows. In Section 2 we present some notations and preliminaries about the Calder\'on-Zygmund operators, the function spaces we use, and a short introduction to wavelets,  a useful tool in our work.  In Section 3 we state our two decomposition theorems (Theorem \ref{new1} and Theorem \ref{Th1}), the $(H^1_b, L^1)$ type estimates for commutators (Theorem \ref{new2}), and some remarks. The bilinear type estimates for commutators of Calder\'on-Zygmund operators (Theorem \ref{bilinear estimates}) and the boundedness of commutators of Calder\'on-Zygmund operators on Hardy spaces are also given in this section.  In Section 4 we give some examples of operators in the class $\mathcal K$ and recall our result from \cite{BGK} which decomposes a product of $f$ in $H^1(\mathbb R^n)$ and $g$ in $BMO(\mathbb R^n)$ as a sum of images by four bilinear operators defined through wavelets. These operators are fundamental  for the two  decomposition theorems. In Section 5 we study the space $H^1_b(\mathbb R^n)$. Section 6 and 7 are devoted to the proofs of the two decomposition theorems,  the $(H^1_b, L^1)$ type estimates  of commutators $[b,T]$ with $T\in \mathcal K$,  and the boundedness results of commutators of Calder\'on-Zygmund operators. Finally, in Section 8  we present without proofs  some results for commutators of fractional integrals.

Throughout the whole paper, $C$ denotes a positive geometric constant which is independent of the main parameters, but may change from line to line.  In $\mathbb R^n$, we denote by $Q=Q[x,r]:= \{y=(y_1,...,y_n)\in\mathbb R^n: \sup_{1\leq i\leq n}|y_i-x_i|\leq r\}$ a cube with center $x=(x_1,...,x_n)$ and radius $r>0$. For any measurable set $E$, we denote by $\chi_E$ its characteristic function, by $|E|$ its  Lebesgue measure, and by $E^c$ the set $\mathbb R^n\setminus E$. For a cube $Q$ and $f$ a locally integrable function, we denote by $f_Q$ the average of $f$ on $Q$.

{\bf Acknowledgements.} The author would like to thank Prof. Aline Bonami  for  many very valuable suggestions, discussions and advices to improve this paper. Specially, Theorem \ref{Th4} is a improvement from the previous version through her ideas. He would also like to thank  Prof. Sandrine Grellier for many helpful suggestions,  her carefully reading and revision of the manuscript. The author is deeply indebted to them. 

\section{Some preliminaries and notations}

As usual, $\mathcal S(\mathbb R^n)$ denotes the Schwartz class of test functions on $\mathbb R^n$, $\mathcal S'(\mathbb R^n)$ the space of tempered distributions, and $C^\infty_0(\mathbb R^n)$ the space of $C^\infty$-functions with compact support.

\subsection{Calder\'on-Zygmund operators}

 Let $\delta\in (0,1]$. A continuous function $K:\mathbb R^n\times \mathbb R^n\setminus\{(x,x):x\in \mathbb R^n\}\to\mathbb C$ is said to be a $\delta$-Calder\'on-Zygmund singular integral kernel if there exists a constant $C>0$  such that
$$|K(x,y)|\leq \frac{C}{|x-y|^n}$$
for all $x\ne y$, and
$$ |K(x,y)-K(x',y)|+|K(y,x)-K(y,x')|\leq C\frac{|x-x'|^\delta}{|x-y|^{n+\delta}}$$
for all $2|x-x'|\leq |x-y|$.

A linear operator $T:\mathcal S(\mathbb R^n)\to\mathcal S'(\mathbb R^n)$ is said to be a $\delta$-Calder\'on-Zygmund operator if $T$ can be extended to a bounded operator on $L^2(\mathbb R^n)$ and if there exists a $\delta$-Calder\'on-Zygmund singular integral kernel $K$ such that for all $f\in C^\infty_0(\mathbb R^n)$ and all $x\notin$ supp $f$, we have
$$Tf(x)=\int_{\mathbb R^n}K(x,y)f(y)dy.$$
We say that $T$ is a Calder\'on-Zygmund operator if it is a $\delta$-Calder\'on-Zygmund operator for some $\delta\in (0,1]$.

 We say that the Calder\'on-Zygmund operator $T$ satisfies the condition $T^*1=0$ (resp., $T1=0$) if $\int_{\mathbb R^n} Ta(x)dx=0$ (resp., $\int_{\mathbb R^n} T^*a(x)dx=0$) holds for all classical $H^1$-atoms $a$. Let $b$ be a locally integrable function on $\mathbb R^n$. We say that the Calder\'on-Zygmund operator $T$ satisfies the condition $T^*b=0$ if $\int_{\mathbb R^n} b(x)Ta(x)dx=0$ holds for all classical $H^1$-atoms $a$.

\subsection{Function spaces}

We first consider the subspace $\mathcal A$ of $\mathcal S(\mathbb R^n)$ defined by
$$\mathcal A=\Big\{\phi\in \mathcal S(\mathbb R^n): |\phi(x)|+ |\nabla\phi(x)|\leq (1+ |x|^2)^{-(n+1)}\Big\},$$
where $\nabla= (\partial/\partial x_1,..., \partial/\partial x_n)$ denotes the gradient. We then define
$$\mathfrak M f(x):= \sup\limits_{\phi\in\mathcal A}\sup\limits_{|y-x|<t}|f*\phi_t(y)|\quad\mbox{and}\quad  \mathfrak mf(x):= \sup\limits_{\phi\in\mathcal A}\sup\limits_{|y-x|<t<1}|f*\phi_t(y)|,$$
where $\phi_t(\cdot)= t^{-n}\phi(t^{-1}\cdot)$. The space $H^1(\mathbb R^n)$ is the space of all tempered distributions $f$ such that $\mathfrak M f\in L^1(\mathbb R^n)$ equipped with the norm $\|f\|_{H^1}= \|\mathfrak M f\|_{L^1}$. The space $h^1(\mathbb R^n)$ denotes the space of all tempered distributions $f$ such that $\mathfrak m f\in L^1(\mathbb R^n)$ equipped with the norm $\|f\|_{h^1}= \|\mathfrak m f\|_{L^1}$. The space $H^{\rm log}(\mathbb R^n)$ (see \cite{Ky, BGK}) denotes the space of all tempered distributions $f$ such that $\mathfrak M f\in L^{\rm log}(\mathbb R^n)$  equipped with the norm $\|f\|_{H^{\rm log}}= \|\mathfrak M f\|_{L^{\rm log}}$. Here  $L^{\rm log}(\mathbb R^n)$ is the space of all measurable functions $f$ such that $\int_{\mathbb R^n}\frac{|f(x)|}{\log(e+|x|)+ \log(e+|f(x)|)}dx<\infty$ with the (quasi-)norm
$$\|f\|_{L^{\rm log}}:=\inf\left\{\lambda>0: \int_{\mathbb R^n}\frac{\frac{|f(x)|}{\lambda}}{\log(e+|x|)+ \log(e+ \frac{|f(x)|}{\lambda})}dx\leq 1\right\}.$$

Clearly, for any $f\in H^1(\mathbb R^n)$, we have
$$\|f\|_{h^1}\leq \|f\|_{H^1}\quad\mbox{and}\quad \|f\|_{H^{\rm log}}\leq \|f\|_{H^1}.$$

We remark that the local real Hardy space $h^1(\mathbb R^n)$, first introduced  by Goldberg \cite{Go}, is larger than $H^1(\mathbb R^n)$ and allows  more flexibility, since global cancellation conditions are not necessary. For example, the Schwartz space is contained in $h^1(\mathbb R^n)$ but not in $H^1(\mathbb R^n)$, and multiplication by cutoff  functions preserves  $h^1(\mathbb R^n)$ but not $H^1(\mathbb R^n)$. Thus it makes $h^1(\mathbb R^n)$ more suitable for working in domains and on manifolds. 

It is well-known (see \cite{FS} or \cite{St}) that  the dual of $H^1(\mathbb R^n)$ is $BMO(\mathbb R^n)$ the space of all  locally integrable functions $f$ with
$$\|f\|_{BMO}:=\sup\limits_{B}\frac{1}{|B|}\int_B |f(x)-f_B|dx<\infty,$$
where the supremum is taken over all balls $B$. We note $\mathbb Q := [0, 1)^n$ and, for $f$ a function in $BMO(\mathbb R^n)$,
$$\|f\|_{BMO^+}:= \|f\|_{BMO}+ |f_{\mathbb Q}|.$$

We should also point out that the space $H^{\rm log}(\mathbb R^n)$ arises naturally in the study of pointwise product of functions in $H^1(\mathbb R^n)$ with functions in $BMO(\mathbb R^n)$, and in the endpoint estimates for the div-curl lemma  (see for example \cite{BFG, BGK, Ky}).

 In \cite{Go} it was  shown that the dual of $h^1(\mathbb R^n)$  can be identified with the space $bmo(\mathbb R^n)$, consisting of locally integrable functions $f$ with
$$\|f\|_{bmo}:= \sup\limits_{|B|\leq 1}\frac{1}{|B|}\int_B |f(x)-f_B|dx+ \sup\limits_{|B|\geq 1}\frac{1}{|B|}\int_B |f(x)|dx<\infty,$$
where the supremums are taken over all balls $B$.

Clearly, for any $f\in bmo(\mathbb R^n)$, we have
$$\|f\|_{BMO}\leq \|f\|_{BMO^+}\leq C \|f\|_{bmo}.$$

We next recall that the space $VMO(\mathbb R^n)$ (resp., $vmo(\mathbb R^n)$) is the closure of $C^\infty_0(\mathbb R^n)$ in $(BMO(\mathbb R^n),\|\cdot\|_{BMO})$ (resp., $(bmo(\mathbb R^n),\|\cdot\|_{bmo})$). It is well-known that (see \cite{CW} and \cite{Daf}) the dual of $VMO(\mathbb R^n)$ (resp., $vmo(\mathbb R^n)$) is the  Hardy space $H^1(\mathbb R^n)$ (resp., $h^1(\mathbb R^n)$). We point out that the space $VMO(\mathbb R^n)$ (resp., $vmo(\mathbb R^n)$) considered by Coifman and Weiss (resp., Dafni \cite{Daf}) is different from the one considered by Sarason. Thanks to  Bourdaud \cite{Bou}, it coincides with the space  $VMO(\mathbb R^n)$ (resp., $vmo(\mathbb R^n)$) considered above.

In the study of the pointwise multipliers for $BMO(\mathbb R^n)$, Nakai and Yabuta  \cite{NY} introduced the space $BMO^{\rm log}(\mathbb R^n)$, consisting of locally integrable functions $f$ with
$$\|f\|_{BMO^{\rm log}}:= \sup\limits_{B(a,r)}\frac{|\log r|+ \log (e+|a|)}{|B(a,r)|}\int_{B(a,r)} |f(x)- f_{B(a,r)}|dx<\infty.$$
There, the authors proved  that a function $g$ is  a pointwise multiplier for $BMO(\mathbb R^n)$ if and only if $g$ belongs to $L^\infty(\mathbb R^n)\cap BMO^{\rm log}(\mathbb R^n)$. Furthermore,  it is also  shown in \cite{Ky} that the space $BMO^{\rm log}(\mathbb R^n)$ is the dual of the space $H^{\rm log}(\mathbb R^n)$.

\begin{Definition}
Let $b$ be a locally integrable function and $1<q\leq \infty$. A function $a$ is called a $(q, b)$-atom related to the cube $Q$ if 

i) supp $a\subset Q$,

ii) $\|a\|_{L^q}\leq |Q|^{1/q-1}$,

iii) $\int_{\mathbb R^n}a(x)dx= \int_{\mathbb R^n}a(x)b(x)dx=0$.
\end{Definition}

The space $\mathcal H^{1,q}_b(\mathbb R^n)$ consists of the subspace of $L^1(\mathbb R^n)$ of functions $f$ which can be written as $f=\sum_{j=1}^\infty \lambda_j a_j$, where $a_j$'s are $(q,b)$-atoms, $\lambda_j \in\mathbb C$, and $\sum_{j=1}^\infty|\lambda_j|<\infty$. As usual, we define on $\mathcal H^{1,q}_b(\mathbb R^n)$ the norm
$$\|f\|_{\mathcal H^{1,q}_b}:= \inf\Big\{\sum_{j=1}^\infty|\lambda_j|: f= \sum_{j=1}^\infty\lambda_j a_j\Big\}.$$

Observe that when $q=\infty$, then the space $\mathcal H^{1,\infty}_b(\mathbb R^n)$ is just the space $\mathcal H^1_b(\mathbb R^n)$ considered in \cite{Pe}. Furthermore, $\mathcal H^{1,\infty}_b(\mathbb R^n)\subset \mathcal H^{1,q}_b(\mathbb R^n)\subset H^1(\mathbb R^n)$ and the inclusions are continuous.

We next introduce the space $H^1_b(\mathbb R^n)$ as follows.

\begin{Definition}
Let $b$ be a non-constant $BMO$-function. The space $H^1_b(\mathbb R^n)$ consists of all $f$ in $H^1(\mathbb R^n)$ such that  $[b,\mathfrak M](f)(x)= \mathfrak M(b(x)f(\cdot)- b(\cdot)f(\cdot))(x)$ belongs to $L^1(\mathbb R^n)$. We equipped  $H^1_b(\mathbb R^n)$ with the norm $\|f\|_{H^1_b}:= \|f\|_{H^1}\|b\|_{BMO}+ \|[b,\mathfrak M](f)\|_{L^1}$.
\end{Definition}

\subsection{Prerequisites on Wavelets}

Let us consider a wavelet basis of $\bR$ with compact support. More explicitly, we are first given a $\C^1(\bR)$-wavelet in Dimension one, called $\psi$, such that $\{2^{j/2}\psi (2^j x-k)\}_{j,k\in \bZ}$ form an $L^2(\bR)$ basis. We assume that this wavelet basis comes for a multiresolution analysis (MRA) on $\bR$, as defined below (see \cite{Me}).
\begin{Definition}
A multiresolution analysis (MRA) on $\mathbb R$ is defined as an increasing sequence $\{V_j\}_{j\in\bZ}$ of closed subspaces of $L^2(\mathbb R)$ with the following four properties

i) $\bigcap_{j\in\mathbb Z}V_j=\{0\}$ and $\overline{\bigcup_{j\in\mathbb Z}V_j}=L^2(\mathbb R)$,

ii) for every $f\in L^2(\mathbb R)$ and every $j\in\mathbb Z$, $f(x)\in V_j$ if and only if $f(2x)\in V_{j+1}$,

iii) for every $f\in L^2(\mathbb R)$ and every $k\in\mathbb Z$, $f(x)\in V_0$ if and only if $f(x-k)\in V_0$,

iv) there exists a function $\phi\in L^2(\mathbb R)$, called the scaling function, such that the family $\{\phi_k(x)=\phi(x-k): k\in\mathbb Z\}$ is an orthonormal basis for $V_0$.
\end{Definition}
It is classical that, given an (MRA) on $\bR$, one can find a wavelet $\psi$ such that $\{2^{j/2}\psi (2^j x-k)\}_{k\in \bZ}$ is an orthonormal basis of $W_j$, the orthogonal complement of $V_j$ in $V_{j+1}$. Moreover, by Daubechies Theorem (see \cite{Dau}), it is possible to find a suitable (MRA) so that $\phi$ and $\psi$ are $\C^1(\bR)$ and compactly supported, $\psi$ has mean $0$ and $\int x\psi (x)dx=0$, which is known as the moment condition. We could content ourselves, in the following theorems, to have $\phi$ and $\psi$ decreasing sufficiently rapidly  at $\infty$, but proofs are simpler with compactly supported wavelets. More precisely we can choose $m>1$ such  that $\phi$ and $\psi$ are supported in the interval $1/2+m(-1/2, +1/2)$, which is obtained from $(0,1)$ by a dilation by $m$ centered at $1/2$.

Going back to $\bR^n$, we recall that a wavelet basis of $\bR^n$ is constructed  as follows. We call $E$ the set $E=\{0,1\}^n\setminus \{(0,\cdots, 0)\}$ and, for $\sigma \in E$, put $\psi^{\sigma}(x)=\phi^{\sigma_1}(x_1)\cdots \phi^{\sigma_n}(x_n)$, with
$\phi^{\sigma_j}(x_j)=\phi(x_j)$ for $\sigma_j =0$ while $\phi^{\sigma_j}(x_j)=\psi(x_j)$ for $\sigma_j =1$. Then the set $\{2^{nj/2}\psi^{\sigma} (2^j x-k)\}_{j\in \bZ, k\in\bZ^n, \sigma\in E}$ is an orthonormal basis of $L^2(\bR^n)$.
As it is classical, for $I$ a dyadic cube of $\bR^n$, which may be written as the set of $x$ such that $2^j x-k \in  (0,1)^n$, we note
$$\psi_I^{\sigma}(x)=2^{nj/2}\psi^{\sigma} (2^j x-k).$$
We also note $\phi_I=2^{nj/2}\phi_{(0,1)^n} (2^j x-k)$, with $\phi_{(0,1)^n}$ the scaling function in $n$ variables, given by
$\phi_{(0,1)^n}(x)=\phi(x_1)\cdots \phi(x_n)$.
In the sequel, the letter $I$ always refers to dyadic cubes. Moreover, we note $kI$  the cube of same center dilated by the coefficient $k$. Because of the assumption on the supports of $\phi$ and $\psi$, the functions
$\psi_I^{\sigma}$ and $\phi_I$ are supported in the cube $mI$.

The wavelet basis $\{\psi_I^\sigma\}$, obtained by letting $I$ vary among dyadic cubes and $\sigma$ in $E$, comes from an (MRA) in $\bR^n$, which we still note $\{V_j\}_{j\in\bZ}$, obtained by taking tensor products of the one-dimensional ones. 

The following theorem gives the wavelet characterization of $H^1(\mathbb R^n)$ (cf. \cite{Me, HW}).

\begin{Theorem}\label{wavelet characterization}
There exists a constant $C>0$ such that $f\in H^1(\mathbb R^n)$ if and only if $\mathcal W_\psi f:= \Big(\sum_I \sum_{\sigma\in E}|\langle f,\psi^\sigma_I\rangle|^2 |I|^{-1}\chi_I\Big)^{1/2}\in L^1(\mathbb R^n)$, moreover,
$$C^{-1}\|f\|_{H^1}\leq \|\mathcal W_\psi f\|_{L^1}\leq C \|f\|_{H^1}.$$
\end{Theorem}

A function $a\in L^2(\mathbb R^n)$ is called a $\psi$-atom related to the (not necessarily dyadic) cube $R$ if it may be written as
$$a=\sum_{I\subset R}\sum_{\sigma\in E} a_{I,\sigma}\psi_I^\sigma$$
with $\|a\|_{L^2}\leq |R|^{-1/2}$. Remark that $a$ is compactly supported in  $mR$ and has mean $0$, so that it is a classical atom related to $mR$, up to the multiplicative constant $m^{n/2}$. It is standard that an atom is in $H^1(\mathbb R^n)$ with norm bounded by a uniform constant. The atomic decomposition gives the converse.
\begin{Theorem}[Atomic decomposition]\label{atomic decomposition}
There exists a constant $C>0$ such that all functions $f \in H^1(\bR^n)$ can be written as the limit in the distribution sense and in $H^1(\mathbb R^n)$ of an infinite sum
$$f =\sum_{j=1}^\infty \lambda_j a_j$$
with $a_j$ $\psi$-atoms related to some dyadic cubes $R_j$ and $\lambda_j$ constants such that
$$C^{-1}\|f\|_{H^1}\leq \sum_{j=1}^\infty |\lambda_j|\leq C \|f\|_{H^1}.$$
\end{Theorem}

This theorem  is a small variation of a standard statement which can be found in \cite{HW}, Section 6.5. Remark that the interest of dealing with finite atomic decompositions has been underlined recently, for instance in \cite{MSV1, Ky}.

 Now, we denote by $H^1_{\rm fin}(\mathbb R^n)$ the vector space of all finite linear combinations of $\psi$-atoms, that is,
$$ f=\sum_{j=1}^k \lambda_j a_j,$$
where $a_j$'s are  $\psi$-atoms. Then, the norm of $f$ in $H^1_{\rm fin}(\mathbb R^n)$ is defined by
$$\|f\|_{H^1_{\rm fin}}=\inf\Big\{\sum_{j=1}^k|\lambda_j|: f= \sum_{j=1}^k \lambda_j a_j\Big\}.$$

By the atomic decomposition theorem, the set $H^1_{\rm fin}(\mathbb R^n)$ is dense in $H^1(\mathbb R^n)$ for the norm $\|\cdot\|_{H^1}$.

The following two  wavelet characterizations of $L^p(\mathbb R^n)$, $1<p<\infty$, and $BMO(\mathbb R^n)$ are well-known (see \cite{Me}).

\begin{Theorem}\label{Lebesgue}
Let $1<p<\infty$. Then the norms $$\|f\|_{L^p},\;\|(\sum_I \sum_{\sigma\in E}|\langle f,\psi_I^\sigma\rangle|^2 |I|^{-1}\chi_I)^{1/2}\|_{L^p}\text{ and }\|(\sum_I \sum_{\sigma\in E}|\langle f,\psi_I^\sigma\rangle|^2 (\psi_I^\sigma)^2)^{1/2}\|_{L^p}$$ are equivalent on $L^p(\mathbb R^n)$.
\end{Theorem}

\begin{Theorem}\label{BMO}
A function $g\in BMO(\mathbb R^n)$ if and only if $$\frac{1}{|R|}\sum_{I\subset R} \sum_{\sigma\in E}|\langle g,\psi_I^\sigma\rangle|^2< \infty$$ for all (not necessarily dyadic)  cubes $R$. Moreover, there exists a constant $C>0$ such that for all $g\in BMO(\mathbb R^n)$,
$$C^{-1}\|g\|_{BMO}\leq \sup_{R}\Big(\frac{1}{|R|}\sum_{I\subset R} \sum_{\sigma\in E}|\langle g,\psi_I^\sigma\rangle|^2\Big)^{1/2}\leq C \|g\|_{BMO},$$
where the supremum is taken over all cubes $R$.
\end{Theorem}

By Theorem \ref{Lebesgue}, Theorem \ref{BMO} and John-Nirenberg inequality, we obtain the following lemma. The proof is easy and will be omitted.
\begin{Lemma}\label{Nirenberg}
Let $f$ be a $\psi$-atom related to the cube $R$ and $b\in BMO(\mathbb R^n)$. Then, $\sum_{I\subset R} \sum_{\sigma\in E}\langle f,\psi_I^\sigma\rangle \langle b,\psi_I^\sigma\rangle (\psi_I^\sigma)^2\in L^{q}(\mathbb R^n)$ for any $q\in (1,2)$.
\end{Lemma}


\section{Bilinear, subbilinear decompositions and  commutators}

Recall that $\mathcal K$ is the set of all sublinear operators $T$ bounded from $H^1(\mathbb R^n)$ into $L^1(\mathbb R^n)$ satisfying 
$$\|(b-b_Q)Ta\|_{L^1}\leq C \|b\|_{BMO},$$
for all $b\in BMO(\mathbb R^n)$, any  $H^1$-atom  $a$ supported in the cube $Q$, where  $C>0$ a constant independent of $b,a$. This class $\mathcal K$ contains almost all important operators in harmonic analysis: Calder\'on-Zygmund type operators, strongly singular integral operators, multiplier operators,  pseudo-differential operators, maximal type operators, the area integral operator of Lusin, Littlewood-Paley type operators, Marcinkiewicz operators,  maximal Bochner-Riesz operators, etc... (See Section 4).

Here and in what follows the bilinear operator $\mathfrak S$ is defined by
$$\mathfrak S(f,g)= - \sum_I \sum_{\sigma\in E}\langle f,\psi_I^\sigma\rangle \langle g,\psi_I^\sigma\rangle (\psi_I^\sigma)^2.$$

In \cite{BGK}, the authors show that $\mathfrak S$ is a bounded bilinear operator from $ H^1(\mathbb R^n)\times BMO(\mathbb R^n)$ into $L^1(\mathbb R^n)$.

\subsection{Two decomposition theorems and $(H^1_b, L^1)$ type estimates}

Let  $b$  be a locally integrable function and $T\in \mathcal K$. As usual, the (sublinear) commutator $[b,T]$ of the operator $T$ is defined by $[b,T](f)(x):= T\Big((b(x)- b(\cdot))f(\cdot)\Big)(x)$.

\begin{Theorem}[Subbilinear decomposition]\label{new1}
Let $T\in \mathcal K$. There exists a bounded subbilinear operator $\mathfrak R= \mathfrak R_T: H^1(\mathbb R^n)\times BMO(\mathbb R^n)\to L^1(\mathbb R^n)$ such that for all $(f,b)\in H^1(\mathbb R^n)\times BMO(\mathbb R^n)$, we have
$$|T(\mathfrak S(f,b))|- \mathfrak R(f,b)\leq |[b, T](f)|\leq \mathfrak R(f,b) + |T(\mathfrak S(f,b))|.$$
\end{Theorem}

\begin{Corollary}\label{weak type}
Let $T\in \mathcal K$ be such that $T$ is of weak type $(1,1)$. Then, the bilinear operator $\mathfrak P(f,g)= [g,T](f)$ maps continuously $H^1(\mathbb R^n)\times BMO(\mathbb R^n)$  into weak-$L^1(\mathbb R^n)$. In particular,  the commutator $[b,T]$ is of weak type $(H^1,L^1)$ if $b\in BMO(\mathbb R^n)$.
\end{Corollary}

We remark that the class of operators $T\in \mathcal K$ of weak type $(1,1)$ contains Calder\'on-Zygmund operators, strongly singular integral operators, multiplier operators, pseudo-differential operators whose symbols in the H\"ormander class $S^m_{\varrho,\delta}$ with $0<\varrho\leq 1, 0\leq \delta< 1, m\leq -n((1-\varrho)/2 +\max\{0, (\delta-\varrho)/2\})$,  maximal type operators, the area integral operator of Lusin, Littlewood-Paley type operators, Marcinkiewicz operators, maximal Bochner-Riesz operators $T^\delta_*$ with $\delta> (n-1)/2$, etc...

When  $T$ is linear and belongs to $\mathcal K$, we obtain the bilinear decomposition for the linear commutator $[b,T]$ of $f$,  $[b,T](f)= bT(f)- T(bf)$, instead of the subbilinear decomposition as stated in Theorem \ref{new1}.

\begin{Theorem}[Bilinear decomposition]\label{Th1}
Let $T$ be a linear operator in $\mathcal K$. Then, there exists a bounded bilinear operator $\mathfrak R= \mathfrak R_T: H^1(\mathbb R^n)\times BMO(\mathbb R^n)\to L^1(\mathbb R^n)$ such that for all $(f,b)\in H^1(\mathbb R^n)\times BMO(\mathbb R^n)$, we have
$$[b, T](f)= \mathfrak R(f,b) + T(\mathfrak S(f,b)).$$
\end{Theorem}

The following result gives $(H^1_b, L^1)$-type estimates for commutators $[b,T]$ when $T$ belongs to the class $\mathcal K$.

\begin{Theorem}\label{new2}
Let  $b$ be a non-constant $BMO$-function and $T\in \mathcal K$. Then, the  commutator $[b, T]$ maps continuously $H^{1}_b(\mathbb R^n)$  into $L^1(\mathbb R^n)$.
\end{Theorem}

Remark that in the particular case of $T$ a $1$-Calder\'on-Zygmund operator and $H^{1}_b(\mathbb R^n)$ replaced by  $\mathcal H^{1}_b(\mathbb R^n)$, P\'erez \cite{Pe} proved
\begin{equation}\label{Grellier}
\sup\{\|[b,T](a)\|_{L^1}: a \,\mbox{is a}\; (\infty,b){\rm -atom}\}<\infty.
\end{equation}
Then he concludes that the (linear) commutator $[b,T]$ maps continuously $\mathcal H^{1}_b(\mathbb R^n)$  into $L^1(\mathbb R^n)$. Notice that $\mathcal H^{1}_b(\mathbb R^n)\subset \mathcal H^{1, q}_b(\mathbb R^n)\subset H^1_b(\mathbb R^n)$, $1<q\leq \infty$, and the inclusions are continuous (see Section 5). However, as mentioned in the introduction, Inequality (\ref{Grellier}) does not suffice to conclude that the  (linear) commutator $[b, T]$ is bounded from $\mathcal H^{1}_b(\mathbb R^n)$ to $L^1(\mathbb R^n)$. We should also  point out  that the $(H^1,L^1)$ weak type estimates  and the $(\mathcal H^1_b, L^1)$ type estimates for the (linear) commutators of multiplier operators (see \cite{ZH, LLM, WL}),  strongly singular Calder\'on-Zygmund operators (see \cite{LinL}) and for the (sublinear) commutators of Littlewood-Paley type operators (see \cite{Liu}), Marcinkiewicz operators (see \cite{LX}), maximal Bochner-Riesz operators (see \cite{LLL, LT, LL})  have been studied recently. However,  the authors just prove Inequality (\ref{Perez 1}) (that is Inequality (\ref{Grellier})) and use Equality (\ref{equality}) which leaves a gap as pointed out in the introduction.


\subsection{Boundedness of linear commutators  on Hardy spaces} 

Analogously to Hardy estimates for bilinear operators of Coifman and Grafakos (\cite{CG}; see also \cite{Do}), we  obtain the following strongly bilinear estimates which improve Corollary \ref{weak type}.
\begin{Theorem}\label{bilinear estimates}
Let $T$ be a linear operator in $\mathcal K$. Assume that $A_i, B_i$, $i=1,..., K$, are Calder\'on-Zygmund operators satisfying $A_i 1=A_i^* 1=B_i 1=B_i^* 1=0$, and for every $f$ and $g$ in $L^2(\mathbb R^n)$, 
$$\int_{\mathbb R^n}\Big(\sum_{i=1}^K A_i f .B_i g\Big)dx=0.$$
Then, the bilinear operator $\mathfrak T$, defined by 
$$\mathfrak T(f,g)=\sum_{i=1}^K  [B_i g,T](A_i f),$$
 maps continuously   $H^1(\mathbb R^n)\times BMO(\mathbb R^n)$ into  $L^1(\mathbb R^n)$.
\end{Theorem}

We now give a sufficient condition for the linear commutator $[b, T]$ to map continuously $H^1_b (\mathbb R^n)$ into $h^1(\mathbb R^n)$.
\begin{Theorem}\label{Th3}
Let  $b$ be a non-constant $BMO^{\rm log}$-function and $T$ be a Calder\'on-Zygmund operator with $T1=T^*1=0$.  Then, the linear commutator $[b,T]$ maps continuously $H^{1}_b(\mathbb R^n)$ into $h^1(\mathbb R^n)$.
\end{Theorem}

The last theorem gives a sufficient condition for the linear commutator $[b, T]$ to map continuously $H^1_b (\mathbb R^n)$ into $H^1(\mathbb R^n)$.
\begin{Theorem}\label{Th4}
Let $b$ be a non-constant $BMO$-function and  $T$ be a Calder\'on-Zygmund operator with $T^*1=T^*b=0$.   Then, the linear commutator $[b,T]$ maps continuously $ H^{1}_b(\mathbb R^n)$ into $H^1(\mathbb R^n)$.
\end{Theorem}

Observe that the condition $T^*b=0$ is  "necessary" in the sense that if the linear commutator $[b,T]$ maps continuously $ H^{1}_b(\mathbb R^n)$ into $H^1(\mathbb R^n)$, then $\int_{\mathbb R^n}b(x) Ta(x)dx=0$ holds for all $(q,b)$-atoms $a$, $1<q\leq \infty$.

Also, let us give some examples to illustrate the sufficient conditions in Theorem \ref{Th4}. To have  many
examples, let us consider  Euclidean spaces $\mathbb R^n, n \geq 2$. Now, consider all Calder\'on-Zygmund operators $T$ such that $T^*1 = 0$. As the closure of $T(H^1(\mathbb R^n))$ is a proper subset of $H^1(\mathbb R^n)$, by the Hahn-Banach theorem (note that $BMO(\mathbb R^n)$ is the dual of $H^1(\mathbb R^n))$, one may take $b$  a non-constant
$BMO$-function such that $\int_{\mathbb R^n}bTa dx=0$ for all $H^1$-atoms $a$, i.e. $T^*b=0$, and thus $b$ and $T$ satisfy the sufficient condition in Theorem \ref{Th4}.

\section{The class $\mathcal K$ and four bilinear operators on $H^1(\mathbb R^n)\times BMO(\mathbb R^n)$}

\subsection{The class $\mathcal K$}
The purpose of this subsection is to give some examples of  operators in the class $\mathcal K$. More precisely,  the class $\mathcal K$ contains almost all important operators in Harmonic analysis: Calder\'on-Zygmund type operators, strongly singular integral operators, multiplier operators, pseudo-differential operators with symbols in the H\"ormander class $S^m_{\varrho,\delta}$ with $0<\varrho\leq 1, 0\leq \delta< 1, m\leq -n((1-\varrho)/2 +\max\{0, (\delta-\varrho)/2\})$ (see \cite{AH,ABKP}), maximal type operators, the area integral operator of Lusin, Littlewood-Paley type operators, Marcinkiewicz operators, maximal Bochner-Riesz operators $T^\delta_*$ with $\delta> (n-1)/2$ (cf. \cite{Le}), etc... It is well-known that these operators $T$ are bounded from $H^1(\mathbb R^n)$ into $L^1(\mathbb R^n)$. So, in order to establish that these ones are  in the class $\mathcal K$, we just need  to show that 
\begin{equation}\label{last}
\|(b-b_Q)Ta\|_{L^1}\leq C \|b\|_{BMO}
\end{equation}
for all $BMO$-function $b$, $H^1$-atom $a$  related to a cube $Q=Q[x_0,r]$ with constant $C>0$ independent of $b,a$. 

Observe that the nontangential grand maximal operator $\mathfrak M$ belongs to $\mathcal K$ since it satisfies Inequality (\ref{last}) (cf. \cite{St}). We refer also to \cite{HST} for the (sublinear) commutators $[b,M_{\varphi,\alpha}]$ of the maximal operators $M_{\varphi,\alpha}$ --note that $M_{\varphi,0}$ lies in $\mathcal K$--.

Here we  just give the proofs for Calder\'on-Zygmund operators (linear operators) and the area integral operator of Lusin (sublinear operator). For the other operators, we leave the proofs to the interested reader.

First recall that $P(x)=\frac{1}{(1+ |x|^2)^{(n+1)/2}}$ is the Poisson kernel and $u_f(x,t):= f*P_t(x)$ is the Poisson integral of $f$. Then the  area integral operator $S$ of Lusin is defined by
$$S(f)(x)=\left(\int_{\Gamma(x)} |\nabla u_f(y,t)|^2 t^{1-n}dydt\right)^{1/2},$$
where $\Gamma(x)$ is the cone $\{(y,t)\in \mathbb R^{n+1}_+: |y-x|<t\}$ with vertex at $x$, while $\nabla u_f= (\partial u_f/\partial x_1,..., \partial u_f/\partial x_1, \partial u_f/\partial t)$ is the gradient of $u_f$ on $\mathbb R^{n+1}_+=\mathbb R^n\times (0,\infty)$.

\begin{Proposition}
Let $\delta\in (0,1]$ and $T$ be a $\delta$-Calder\'on-Zygmund operator. Then $T$ satisfies Inequality (\ref{last}), and thus $T$ belongs to $ \mathcal K$.
\end{Proposition}

\begin{proof}
We cut the integral of $|(b-b_Q)Ta|$ into two parts. By Schwarz inequality and the boundedness of $T$ on $L^2(\mathbb R^n)$, we have
\begin{eqnarray*}
\int_{2Q}|b(x)-b_Q||Ta(x)|dx
&\leq&
C \left(\int_{2Q}|b(x)-b_Q|^2dx\right)^{1/2} \|a\|_{L^2}\\
&\leq&
C \|b\|_{BMO}
\end{eqnarray*}
here one used the fact $|b_{2Q}-b_Q|\leq C \|b\|_{BMO}$. Next, for $x\notin 2Q$, 
\begin{eqnarray*}
|Ta(x)|
&=& \left|\int_{Q} (K(x,y)- K(x,x_0)) a(y)dy\right|\\
&\leq&
C \int_{Q}\frac{|y-x_0|^\delta}{|x-x_0|^{n+\delta}}|a(y)|dy\\
&\leq&
C \frac{r^\delta}{|x-x_0|^{n+\delta}}.
\end{eqnarray*}
Therefore, 
$$\int_{(2Q)^c}|b(x)- b_Q||Ta(x)|dx\leq C\int_{Q^c} |b(x)- b_Q|\frac{r^\delta}{|x-x_0|^{n+\delta}}dx\leq C\|b\|_{BMO},$$
since the last inequality is classical (cf. \cite{St}). This finishes the proof.

\end{proof}

\begin{Corollary}\label{Riesz transforms}
Let $\mathcal R_j, j=1,...,n$, be the classical Riesz transforms. Then, $\mathcal R_j$ belongs to $\mathcal K$ for all $j=1,...,n$.
\end{Corollary}

\begin{Proposition}
The area integral operator $S$  satisfies Inequality (\ref{last}), and thus $S$ belongs to $ \mathcal K$.
\end{Proposition}
\begin{proof}
We also cut the integral of $|(b-b_Q)S(a)|$ into two parts. By Schwarz inequality and the boundedness of $S$ on $L^2(\mathbb R^n)$, we have
\begin{eqnarray*}
\int_{2Q}|b(x)-b_Q||S(a)(x)|dx
&\leq&
C \left(\int_{2Q}|b(x)-b_Q|^2dx\right)^{1/2} \|a\|_{L^2}\\
&\leq&
C \|b\|_{BMO}.
\end{eqnarray*}
 Next, for $x\notin 2Q$, by using the equality
$$u_a(y,t)=\int_{\mathbb R^n}\frac{1}{t^n}\Big(P\Big(\frac{y-z}{t}\Big)- P\Big(\frac{y-x_0}{t}\Big)\Big)a(z)dz,$$
since $\int_{\mathbb R^n}a(z)dz=0$, it is easy to establish that 
$$S(a)(x)=\left(\int_{\Gamma(x)} |\nabla u_a(y,t)|^2 t^{1-n}dydt\right)^{1/2}\leq C \frac{r}{|x-x_0|^{n+1}}.$$
Therefore, 
$$\int_{(2Q)^c}|b(x)- b_Q||S(a)(x)|dx\leq C\int_{Q^c} |b(x)- b_Q|\frac{r}{|x-x_0|^{n+1}}dx\leq C\|b\|_{BMO},$$
which ends the proof.
\end{proof}

We should point out that the Littlewood-Paley type operators can be viewed as vector-valued Calder\'on-Zygmund operators (see \cite{RRT}). See also \cite{HST}  in the context of vector-valued commutators.

\subsection{Four bilinear operators on $H^1(\mathbb R^n)\times BMO(\mathbb R^n)$} 

We now consider four bilinear operators on $H^1(\mathbb R^n)\times BMO(\mathbb R^n)$ which are fundamental for our bilinear decomposition theorem.

We first state some lemmas whose proofs can be found in \cite{BGK}.

\begin{Lemma}\label{Le1}
The bilinear operator $\Pi_3$ defined on $H^1(\mathbb R^n)\times BMO(\mathbb R^n)$ by
$$\Pi_3(f,g)= \sum_{I}\sum_{\sigma\in E}\langle f, \psi^\sigma_I\rangle \langle g, \psi^\sigma_I\rangle (\psi^\sigma_I)^2$$
 is a bounded bilinear operator from $H^1(\mathbb R^n)\times BMO(\mathbb R^n)$ into $L^1(\mathbb R^n)$. 
\end{Lemma}

Observe that $\mathfrak S(f,g)= -\Pi_3(f,g)$ for all $(f,g)\in H^1(\mathbb R^n)\times BMO(\mathbb R^n)$.

\begin{Lemma}\label{Le2}
The bilinear operator $\Pi_4$, defined on $H^1(\mathbb R^n)\times BMO(\mathbb R^n)$ by
$$\Pi_4(f,g)= \sum_{I, I'}\sum_{\sigma,\sigma'\in E}\langle f, \psi^\sigma_I\rangle \langle g, \psi^{\sigma'}_{I'}\rangle \psi^\sigma_I \psi^{\sigma'}_{I'},$$
the sums being taken over all dyadic cubes $I, I'$ and $\sigma,\sigma'\in E$ such that $(I,\sigma)\ne (I',\sigma')$, is a bounded bilinear operator from $H^1(\mathbb R^n)\times BMO(\mathbb R^n)$ into $H^1(\mathbb R^n)$.
\end{Lemma}

\begin{Lemma}\label{Le3}
The bilinear operator $\Pi_1$ defined by
$$\Pi_1(a, g)=\sum_{|I|=|I'|}\sum_{\sigma\in E}\langle a, \phi_I\rangle \langle g, \psi^\sigma_{I'}\rangle \phi_I \psi^\sigma_{I'},$$
where $a$ is a $\psi$-atom and $g\in BMO(\mathbb R^n)$, can be extended into a bounded bilinear operator from $H^1(\mathbb R^n)\times BMO(\mathbb R^n)$ into $H^1(\mathbb R^n)$.
\end{Lemma}

\begin{Lemma}\label{Le4}
The bilinear operator $\Pi_2$ defined by
$$\Pi_2(a, g)=\sum_{|I|=|I'|}\sum_{\sigma\in E}\langle a, \psi^\sigma_I \rangle \langle g, \phi_{I'}\rangle \psi^\sigma_I \phi_{I'},$$
where $a$ is a $\psi$-atom related to the cube $R$ and $g\in BMO(\mathbb R^n)$, can be extended into a bounded bilinear operator from $H^1(\mathbb R^n)\times BMO^{+}(\mathbb R^n)$ into $H^{\rm log}(\mathbb R^n)$. Furthermore, we can write
\begin{equation}\label{Tri}
    \Pi_2(a,g)= h^{(1)}+ \kappa g_R h^{(2)}
\end{equation}
where $\|h^{(1)}\|_{H^1}\leq C   \|g\|_{\BMO}$, $h^{(2)}$ is an atom related to $mR$, and $\kappa$ a uniform constant, independent of $a$ and $g$.
\end{Lemma}

The following remarks are useful in our proofs in Section 6 and Section 7.

\begin{Remark}\label{remark}
\begin{enumerate}
\item If $g\in BMO(\mathbb R^n)$ and $f\in H^1(\mathbb R^n)$ such that $fg\in L^1(\mathbb R^n)$, then
$$\int_{\mathbb R^n} fgdx= -\int_{\mathbb R^n}\mathfrak S(f,g)dx= \sum_{I}\sum_{\sigma\in E}\langle f, \psi^\sigma_I\rangle \langle g, \psi^\sigma_I\rangle.$$

\item For any $(f,g)\in H^1(\mathbb R^n)\times BMO(\mathbb R^n)$ and $c$  a constant, we have
$$\Pi_i(f,g)= \Pi_i(f,g+c)\,, \, i=1,3,4.$$

\item As a consequence of Lemma \ref{Le4},  if $g_R=0$ then Equality (\ref{Tri}) gives that $\Pi_2(a,g)\in H^1(\mathbb R^n)$. Moreover, $\|\Pi_2(a,g)\|_{H^1}\leq C\|g\|_{BMO}$.
\end{enumerate}
\end{Remark}

In \cite{BGK}, the authors have shown the  following decomposition theorem for the product space $H^1(\mathbb R^n)\times BMO(\mathbb R^n)$.

\begin{Theorem}[Decomposition theorem]\label{Decomposition theorem}
Let  $f\in H^1(\mathbb R^n)$ and $g\in BMO(\mathbb R^n)$. Then, we have the following  decomposition
$$fg=\Pi_1(f,g)+ \Pi_2(f,g)+ \Pi_3(f,g)+ \Pi_4(f,g),$$
that is
$$fg=\Pi_1(f,g)+ \Pi_2(f,g)+ \Pi_4(f,g) - \mathfrak S(f,g).$$
\end{Theorem}

\section{The space $H^1_b(\mathbb R^n)$}

Let $b$ be a non-constant $BMO$-function. In this section, we study the space $H^1_b(\mathbb R^n)$. In particular, we give  some characterizations of the space $H^1_b(\mathbb R^n)$ (see Theorem \ref{Bonami}), and  the comparison with the space $\mathcal H^1_b(\mathbb R^n)$ of P\'erez (see Theorem \ref{compare}).

First, let us consider the class $\widetilde{\mathcal K}$ of all $T\in\mathcal K$ such that $T$  characterizes the space $H^1(\mathbb R^n)$, that means $f\in H^1(\mathbb R^n)$ if and only if $Tf\in L^1(\mathbb R^n)$. Clearly, the class $\widetilde{\mathcal K}$ contain the maximal operator $\mathfrak M$, the area integral operator $S$ of Lusin, the Littlewood-Paley $g$-operator (see \cite{FS}), the Littlewood-Paley $g_\lambda^*$-operator  with $\lambda> 3n$ (see \cite{HMY}), etc...

Here and in what  follows, the symbol $f\approx g$ means that $C^{-1}f\leq g\leq C f$ for some constant $C>0$. We obtain the following characterization of $H^1_b(\mathbb R^n)$.

\begin{Theorem}\label{Bonami}
Let $b$ be a non-constant $BMO$-function and $T\in \widetilde{\mathcal K}$. For $f\in H^1(\mathbb R^n)$, the following conditions are equivalent:

i) $f\in H^1_b(\mathbb R^n)$.

ii) $\mathfrak S(f,b)\in H^1(\mathbb R^n)$.

iii)  $[b, \mathcal R_j](f)\in L^1(\mathbb R^n)$ for all $j=1,...,n$.

iv)  $[b,T](f)\in L^1(\mathbb R^n)$.

Furthermore, if one of these conditions is satisfied, then
\begin{eqnarray*}
\|f\|_{H^1_b}&=& \|f\|_{H^1}\|b\|_{BMO}+ \|[b,\mathfrak M](f)\|_{L^1}\\
&\approx& \|f\|_{H^1}\|b\|_{BMO}+ \|\mathfrak S(f,b)\|_{H^1}\\
&\approx& \|f\|_{H^1}\|b\|_{BMO}+ \sum_{j=1}^n \|[b, \mathcal R_j](f)\|_{L^1}\\
&\approx& \|f\|_{H^1}\|b\|_{BMO}+ \|[b,T](f)\|_{L^1},
\end{eqnarray*}
where the constants are independent of $f$ and $b$.
\end{Theorem}

\begin{Remark}\label{the largest space}
Theorem \ref{new2} and Theorem \ref{Bonami} give that  $[b,T]$ is bounded from $H^1_b(\mathbb R^n)$ to $L^1(\mathbb R^n)$ for every $T$ a Calder\'on-Zygmund singular integral operator. Furthermore, $H^1_b(\mathbb R^n)$  is the largest space having this property.
\end{Remark}

\begin{proof}[Proof of Theorem \ref{Bonami}]
$(i)\Leftrightarrow (ii)$ By Theorem \ref{new1}, there exists a bounded subbilinear operator $\mathfrak R: H^1(\mathbb R^n)\times BMO(\mathbb R^n)\to L^1(\mathbb R^n)$ such that
$$\mathfrak M(\mathfrak S(f,b))- \mathfrak R(f,b)\leq |[b, \mathfrak M](f)|\leq \mathfrak R(f,b) + \mathfrak M(\mathfrak S(f,b)).$$
Consequently, $\mathfrak S(f,b)\in H^1(\mathbb R^n)$ if and only if $[b, \mathfrak M](f)\in L^1(\mathbb R^n)$. Moreover,
$$ \|f\|_{H^1_b}\approx \|f\|_{H^1}\|b\|_{BMO}+  \|\mathfrak S(f,b)\|_{H^1}.$$

 $(ii)\Leftrightarrow (iii)$.  By Theorem \ref{Th1}, there exist $n$ bounded bilinear operators $\mathfrak R_j: H^1(\mathbb R^n)\times BMO(\mathbb R^n)\to L^1(\mathbb R^n)$, $j=1,...,n$, such that
$$[b, \mathcal R_j](f)= \mathfrak R_j(f,b)+ \mathcal R_j(\mathfrak S(f,b)).$$
Consequently, $\mathfrak S(f,b)\in H^1(\mathbb R^n)$ if and only if $[b, \mathcal R_j](f)\in L^1(\mathbb R^n)$ for all $j=1,...,n$. Moreover,

$$ \|f\|_{H^1}\|b\|_{BMO}+  \|\mathfrak S(f,b)\|_{H^1}\approx \|f\|_{H^1}\|b\|_{BMO}+ \sum_{j=1}^n \|[b, \mathcal R_j](f)\|_{L^1}.$$

$(ii)\Leftrightarrow (iv)$.  By Theorem \ref{new1}, there exists a bounded subbilinear operator $\mathfrak R: H^1(\mathbb R^n)\times BMO(\mathbb R^n)\to L^1(\mathbb R^n)$ such that
$$|T(\mathfrak S(f,b))|- \mathfrak R(f,b)\leq |[b, T](f)|\leq \mathfrak R(f,b) + |T(\mathfrak S(f,b))|.$$
Consequently, $\mathfrak S(f,b)\in H^1(\mathbb R^n)$ if and only if $[b, T](f)\in L^1(\mathbb R^n)$ since $T\in \widetilde{\mathcal K}$. Moreover,
$$\|f\|_{H^1}\|b\|_{BMO}+  \|\mathfrak S(f,b)\|_{H^1}\approx \|f\|_{H^1}\|b\|_{BMO}+  \|[b,T](f)\|_{L^1}.$$

\end{proof}

Remark  that the constants in the last equivalence depend on $T$.

The following lemma is an immediate corollary of the weak convergence theorem in $H^1(\mathbb R^n)$ of Jones and Journ\'e. See also \cite{Daf} in the setting of $h^1(\mathbb R^n)$.

\begin{Lemma}\label{weak convergence}
Let $\{f_k\}_{k\geq 1}$ be a bounded sequence  in $H^1(\mathbb R^n)$ (resp., in $h^1(\mathbb R^n)$) such that $f_k$ tends to $f$ in $L^1(\mathbb R^n)$. Then $f$ in $H^1(\mathbb R^n)$ (resp., in $h^1(\mathbb R^n)$), and
$$\|f\|_{H^1}\leq \varliminf_{k\to\infty}\|f_k\|_{H^1}\quad(resp., \|f\|_{h^1}\leq \varliminf_{k\to\infty}\|f_k\|_{h^1}).$$
\end{Lemma}

\begin{Theorem}\label{compare}
Let $b$ be a non-constant $BMO$-function and $1<q\leq \infty$. Then, $\mathcal H^{1,q}_b(\mathbb R^n)\subset H^1_b(\mathbb R^n)$ and the inclusion is continuous.
\end{Theorem}

\begin{proof}
Let $a$ be a $(q,b)$-atom related to the cube $Q$. We first prove that $(b-b_{Q})a$ is $C \|b\|_{BMO}$ times a classical $(\widetilde q+1)/2$-atom.
One has supp $(b-b_{Q})a\subset$ supp $a\subset Q$ and $\int_{\mathbb R^n}(b(x)-b_{Q})a(x)dx=\int_{\mathbb R^n}b(x)a(x)dx- b_{Q}\int_{\mathbb R^n}a(x)dx=0$. Moreover, by H\"older inequality and John-Nirenberg inequality, we get
$$\|(b-b_{Q})a\|_{L^{(\widetilde q+1)/2}}\leq \|(b-b_{Q})\chi_{Q}\|_{L^{\widetilde q(\widetilde q+1)/(\widetilde q-1)}}\|a\|_{L^{\widetilde q}} \leq C \|b\|_{BMO}|Q|^{(-\widetilde q+1)/(\widetilde q+1)},$$
where $\widetilde q=q$ if $1<q<\infty$, $\widetilde q=2$ if $q=\infty$, and  $C>0$ is independent of $b,a$. Hence, $(b-b_{Q})a$ is $C \|b\|_{BMO}$ times a classical $(\widetilde q+1)/2$-atom, and $\|(b-b_{Q})a\|_{H^1}\leq C\|b\|_{BMO}$.

We now prove that $\mathfrak S(a,b)$ belongs to $H^1$.

By Theorem \ref{Th1},  there exist $n$ bounded bilinear operators $\mathfrak R_j: H^1(\mathbb R^n)\times BMO(\mathbb R^n)\to L^1(\mathbb R^n)$, $j=1,...,n$, such that
$$[b, \mathcal R_j](a)= \mathfrak R_j(a,b)+ \mathcal R_j(\mathfrak S(a,b)),$$
since  $\mathcal R_j$ is linear and belongs to $\mathcal K$ (see Corollary \ref{Riesz transforms}). Consequently, for all $j=1,...,n$, as $\mathcal R_j\in \mathcal K$,
\begin{eqnarray*}
\|\mathcal R_j(\mathfrak S(a,b))\|_{L^1}&=& \|(b-b_Q)\mathcal R_j(a)- \mathcal R_j((b-b_Q)a)- \mathfrak R_j(a,b)\|_{L^1}\\
&\leq& \|(b-b_Q)\mathcal R_j(a)\|_{L^1}+  \|\mathcal R_j\|_{H^1\to L^1}\|((b-b_Q)a)\|_{H^1}+ \|\mathfrak R_j(a,b)\|_{L^1}\\
&\leq& C \|b\|_{BMO}.
\end{eqnarray*}
This proves that $\mathfrak S(a,b)\in H^1(\mathbb R^n)$ since $\|\mathfrak S(a,b)\|_{L^1}\leq C\|b\|_{BMO}$, and moreover that
\begin{equation}\label{Perez}
\|\mathfrak S(a,b)\|_{H^1}\leq C\|b\|_{BMO}.
\end{equation}

Now, for any $f\in \mathcal H^{1,q}_b(\mathbb R^n)$, there exists an expansion $f=\sum_{j=1}^\infty \lambda_j a_j$ where the $a_j$'s are $(q,b)$-atoms and  $\sum_{j=1}^\infty |\lambda_j|\leq 2 \|f\|_{\mathcal H^{1,q}_b}$. Then the sequence $\{\sum_{j=1}^k \lambda_j a_j\}_{k\geq 1}$ converges to $f$ in $\mathcal H^{1,q}_b(\mathbb R^n)$ and thus in $H^1(\mathbb R^n)$. Hence, Lemma \ref{Le1} implies that the sequence $\Big\{\mathfrak S\Big(\sum_{j=1}^k \lambda_j a_j, b\Big)\Big\}_{k\geq 1}$ converges to $\mathfrak S(f,b)$ in $L^1(\mathbb R^n)$. In addition, by (\ref{Perez}),
$$\Big\|\mathfrak S\Big(\sum_{j=1}^k \lambda_j a_j, b\Big)\Big\|_{H^1}\leq \sum_{j=1}^k |\lambda_j| \|\mathfrak S(a_j, b)\|_{H^1}\leq C \|f\|_{\mathcal H^{1,q}_b}\|b\|_{BMO}.$$
We then use Lemma \ref{weak convergence} to conclude that $\mathfrak S(f,b)\in H^1(\mathbb R^n)$, and thus $f\in H^1_b(\mathbb R^n)$ (see Theorem \ref{Bonami}). Moreover,
\begin{eqnarray*}
\|f\|_{H^1_b}&\leq& C (\|f\|_{H^1}\|b\|_{BMO}+ \|\mathfrak S(f,b)\|_{H^1})\\
&\leq& C \Big(\|f\|_{\mathcal H^{1,q}_b}\|b\|_{BMO}+ \varliminf\limits_{k\to \infty}\Big\|\mathfrak S\Big(\sum_{j=1}^k \lambda_j a_j, b\Big)\Big\|_{H^1}\Big)\\
&\leq& C \|f\|_{\mathcal H^{1,q}_b}\|b\|_{BMO},
\end{eqnarray*}
which ends the proof.
\end{proof}

From Theorem \ref{new2} and Theorem \ref{Bonami}, we get the following corollary.
\begin{Corollary}\label{tui}
Let $b$ be a $BMO$-function, $T\in \mathcal K$ and $1<q\leq \infty$. Then the linear commutator $[b,T]$ maps continuously $\mathcal H^{1,q}_b(\mathbb R^n)$ into $L^1(\mathbb R^n)$.
\end{Corollary}

\section{Proof of  Theorem \ref{new1}, Theorem \ref{Th1}, Theorem \ref{new2}}

In order to prove the decomposition theorems (Theorem \ref{Th1} and Theorem \ref{new1}), we need the following two lemmas.
\begin{Lemma}\label{vo}
Let $T\in \mathcal K$ and $a$ be a classical $H^1$-atom related to the cube $mQ$. Then, there exists a positive constant $C= C(m)$ such that
$$\|(g-g_Q)Ta\|_{L^1}\leq C \|g\|_{BMO},\; \mbox{for all}\; g\in BMO(\mathbb R^n).$$
\end{Lemma}

\begin{proof}
Since $T\in \mathcal K$ and since $|g_Q-g_{mQ}|\leq C(m) \|g\|_{BMO}$, we have
$$\|(g-g_Q)Ta\|_{L^1}\leq C(m)\|g\|_{BMO}\|Ta\|_{L^1}+ \|(g-g_{mQ})Ta\|_{L^1}\leq C \|g\|_{BMO}.$$
\end{proof}

\begin{Lemma}\label{finite decomposition}
The norms $\|\cdot\|_{H^1}$ and $\|\cdot\|_{H^1_{\rm fin}}$ are equivalent on $H^1_{\rm fin}(\mathbb R^n)$.
\end{Lemma}

We point out that in the proof below we use the results and notations of Theorem 5.12 of \cite{HW}. Even though the proofs in \cite{HW} are in the one-dimensional case, they can be easily carried out in higher dimension as well.

\begin{proof}[The proof of Lemma \ref{finite decomposition}]
Obviously, $H^1_{\rm fin}(\mathbb R^n)\subset H^1(\mathbb R^n)$ and for all $f\in H^1_{\rm fin}(\mathbb R^n)$, we have $\|f\|_{H^1}\leq C \|f\|_{H^1_{\rm fin}}$. We now have to show that there exists a constant $C>0$ such that  for all $f\in H^1_{\rm fin}(\mathbb R^n)$,
$$\|f\|_{H^1_{\rm fin}}\leq C \|f\|_{H^1}.$$

By homogeneity, we can assume that $\|f\|_{H^1}=1$. We write $f=\sum_{j=1}^{N_0}\lambda_j a_j$, where the $a_j$'s are $\psi$-atoms related to the cubes $R_j$'s. Since $f\in L^2(\mathbb R^n)\cap H^1(\mathbb R^n)$, there exists a $\psi$-atomic decomposition (see \cite{HW}, Theorem 5.12)
\begin{eqnarray*}
f= \sum_{I}\sum_{\sigma\in E}\langle f, \psi^\sigma_I\rangle \psi^\sigma_I
&=& \sum_{k\in\mathbb Z}\sum_{i\in \Lambda_k}\left(\sum_{{I\subset \widetilde I^i_k}, I\in \mathcal B_k}\sum_{\sigma\in E} \langle f, \psi^\sigma_I\rangle \psi^\sigma_I\right)
\end{eqnarray*}
where $\sum_{{I\subset \widetilde I^i_k}, I\in \mathcal B_k}\sum_{\sigma\in E} \langle f, \psi^\sigma_I\rangle \psi^\sigma_I= \lambda(k,i) a_{k,i}$ with $a_{k,i}$  $\psi$-atoms related to the cubes $m\widetilde I^i_k$ and
\begin{equation}\label{KY}
\sum_{k\in\mathbb Z}\sum_{i\in \Lambda_k}|\lambda(k,i)|\leq C \|f\|_{H^1}=C.
\end{equation}
We note that supp $a_{k,i}\subset \bigcup_{j=1}^{N_0} mR_j$ for all  $k\in\mathbb Z, i\in \Lambda_k$. Recall that
\begin{eqnarray*}
\mathcal W_\psi f
&=&\Big(\sum_{I}\sum_{\sigma\in E}|\langle f, \psi^\sigma_I\rangle|^2 |I|^{-1}\chi_I\Big)^{1/2}\\
&=&\Big(\sum_{j=1}^{N_0}\sum_{I\subset R_j}\sum_{\sigma\in E}|\langle f, \psi^\sigma_I\rangle|^2 |I|^{-1}\chi_I\Big)^{1/2}
\end{eqnarray*}
and $\Omega_k= \{x\in\mathbb R^n: \mathcal W_\psi f(x)> 2^k\}$ for any $k\in\mathbb Z$. Clearly, supp $\mathcal W_\psi f\subset \bigcup_{j=1}^{N_0} mR_j$. So, there exists a cube $Q$ such that $\Omega_k\subset\,\mbox{supp}\,\mathcal W_\psi f\subset \bigcup_{j=1}^{N_0} mR_j\subset Q$ for all $k\in\mathbb Z$. We now denote by $k'$ the largest integer $k$ such that $2^k\leq |Q|^{-1}$. Then, we define the functions $g$ and $\ell$ by
$$g= \sum_{k\leq k'}\sum_{i\in \Lambda_k}\left(\sum_{{I\subset \widetilde I^i_k}, I\in \mathcal B_k}\sum_{\sigma\in E} \langle f, \psi^\sigma_I\rangle \psi^\sigma_I\right)\;\mbox{and}\; \ell= \sum_{k> k'}\sum_{i\in \Lambda_k}\left(\sum_{{I\subset \widetilde I^i_k}, I\in \mathcal B_k}\sum_{\sigma\in E} \langle f, \psi^\sigma_I\rangle \psi^\sigma_I\right).$$

Obviously, $f=g+\ell$, moreover, supp $g\subset Q$ and supp $\ell\subset Q$. On the other hand, it follows from Theorem 5.12 of  \cite{HW} that $\sum_{{I\subset \widetilde I^i_k}, I\in \mathcal B_k}\sum_{\sigma\in E} |\langle f, \psi^\sigma_I\rangle|^2\leq C 2^{2k} |\widetilde I^i_k\cap \Omega_k|$. Hence, as the dyadic cubes $\widetilde I^i_k$ are disjoint (see also \cite{HW}), we get
\begin{eqnarray*}
\|g\|_{L^2}^2
&\leq&
 C \sum_{k\leq k'}\sum_{i\in \Lambda_k}\sum_{{I\subset \widetilde I^i_k}, I\in \mathcal B_k}\sum_{\sigma\in E} |\langle f, \psi^\sigma_I\rangle|^2\\
&\leq&
C \sum_{k\leq k'}\sum_{i\in \Lambda_k} 2^{2k} |\widetilde I^i_k\cap \Omega_k|\leq C \sum_{k\leq k'} 2^{2k} |\Omega_k|\\
&\leq&
C 2^{2k'}|Q|\leq C |Q|^{-1}.
\end{eqnarray*}
This proves that $C^{-1/2}g$ is a $\psi$-atom related to the cube $Q$.

Now, for any positive integer $K$, set $F_K= \{(k,i): k> k', |k|+ |i|\leq K\}$ and $\ell_K= \sum_{(k,i)\in F_K}\left(\sum_{{I\subset \widetilde I^i_k}, I\in \mathcal B_k}\sum_{\sigma\in E} \langle f, \psi^\sigma_I\rangle \psi^\sigma_I\right)$. Observe that since $f\in L^2(\mathbb R^n)$, the series $\sum_{k> k'}\sum_{i\in \Lambda_k}\left(\sum_{{I\subset \widetilde I^i_k}, I\in \mathcal B_k}\sum_{\sigma\in E} \langle f, \psi^\sigma_I\rangle \psi^\sigma_I\right)$ converges in $L^2(\mathbb R^n)$. So, for any $\varepsilon>0$, if $K$ is large enough,  $\varepsilon^{-1}(\ell- \ell_K)$ is a $\psi$-atom related to the cube $Q$. Therefore, $f= g+ \ell_K + (\ell- \ell_K)$ is a finite linear  combination of atoms for $f$, and thus
\begin{eqnarray*}
\|f\|_{H^1_{\rm fin}}
&\leq& C(\|g\|_{H^1_{\rm fin}}+ \|\ell_K\|_{H^1_{\rm fin}}+ \|\ell- \ell_K\|_{H^1_{\rm fin}})\\
&\leq& 
C\Big(C+ \sum_{k\in\mathbb Z}\sum_{i\in \Lambda_k}|\lambda(k,i)|+ \varepsilon\Big)\leq C
\end{eqnarray*}
by (\ref{KY}). It ends the proof.

\end{proof}

\begin{proof}[Proof of Theorem \ref{new1}]
We define the subbilinear operator $\mathfrak R$  by
$$\mathfrak R(f,b)(x):= \Big|T\Big(b(x)f(\cdot)- \Pi_2(f,b)(\cdot)\Big)(x)\Big|+ |T(\Pi_1(f,b))(x)|+ |T(\Pi_4(f,b))(x)|$$
for all $(f,b)\in H^1(\mathbb R^n)\times BMO(\mathbb R^n)$. Then, by Theorem \ref{Decomposition theorem}, we obtain that 
$$|T(\mathfrak S(f,b))| -\mathfrak R(f,b)\leq |[b,T](f)|\leq \mathfrak R(f,b)+ |T(\mathfrak S(f,b))|.$$

By Lemma \ref{Le1}, Lemma \ref{Le2} and Lemma \ref{Le3}, it is sufficient to  show that the subbilinear operator
$$\mathfrak U(f,b)(x):= \Big|T\Big(b(x)f(\cdot)- \Pi_2(f,b)(\cdot)\Big)(x)\Big|$$
 is  bounded from $H^1(\mathbb R^n)\times BMO(\mathbb R^n)$ into $L^1(\mathbb R^n)$. 

We first consider $b\in BMO(\mathbb R^n)$ and $f$  a $\psi$-atom related to the cube $Q$. Then, by Remark \ref{remark}, we have
$$\mathfrak U(f,b)(x)= \mathfrak U(f,b- b_Q)(x)\leq |(b(x)-b_Q)Tf(x)|+ |T(\Pi_2(f, b-b_Q))(x)|.$$

Consequently, by Remark \ref{remark}, Lemma \ref{vo} and the fact $f$ is $C$ times a classical atom related to the cube $mQ$, we obtain that
\begin{equation}\label{con}
\|\mathfrak U(f,b)\|_{L^1}\leq \|(b-b_Q)Tf\|_{L^1}+ \|T\|_{H^1\to L^1}\|\Pi_2(f, b-b_Q)\|_{H^1}\leq C \|b\|_{BMO},
\end{equation}
where $C>0$ independent of $f,b$.

Now, let $b\in BMO(\mathbb R^n)$ and $f\in H^1_{\rm fin}(\mathbb R^n)$. By Lemma \ref{finite decomposition}, there exists a finite decomposition $f=\sum_{j=1}^k \lambda_j a_j$ such that $\sum_{j=1}^k |\lambda_j|\leq C \|f\|_{H^1}$. Consequently, by (\ref{con}), we obtain that
$$\|\mathfrak U(f,b)\|_{L^1}\leq \sum_{j=1}^k |\lambda_j| \|\mathfrak U(a_j,b)\|_{L^1}\leq  C \|f\|_{H^1}\|b\|_{BMO},$$
which ends the proof as $H^1_{\rm fin}(\mathbb R^n)$ is dense in $H^1(\mathbb R^n)$ for the norm $\|\cdot\|_{H^1}$.

\end{proof}

\begin{proof}[Proof of Theorem \ref{Th1}]
We define the bilinear operator $\mathfrak R$ by
$$\mathfrak R(f,b)=  \Big(b Tf- T(\Pi_2(f,b))\Big) - T(\Pi_1(f,b)+ \Pi_4(f,b)),$$
for all $(f,b)\in H^1(\mathbb R^n)\times BMO(\mathbb R^n)$. Then, it follows from Theorem \ref{Decomposition theorem} and the proof of Theorem \ref{new1} that
$$[b, T](f)= \mathfrak R(f,b)+ T(\mathfrak S(f,b)),$$
where the bilinear operator $\mathfrak R$ is  bounded from $H^1(\mathbb R^n)\times BMO(\mathbb R^n)$ into $L^1(\mathbb R^n)$. This completes the proof.

\end{proof}

\begin{proof}[Proof of Theorem \ref{new2}]
Theorem \ref{new2} is an immediate corollary of Theorem \ref{new1} and Theorem \ref{Bonami}.
\end{proof}

\section{Proof of Theorem \ref{bilinear estimates}, Theorem \ref{Th3} and Theorem \ref{Th4}}

First we recall the following well-known result.\\\\
{\bf Theorem A.} (see \cite{CG} or \cite{Do})
Let $T$ be a Calder\'on-Zygmund operator satisfying $T1=T^*1=0$, $1<q<\infty$ and $1/p+1/q=1$. Then, $fTg- gT^*f\in H^1(\mathbb R^n)$ for all $f\in L^p(\mathbb R^n)$, $g\in L^q(\mathbb R^n)$.

Now, in order to prove the bilinear type estimates and the Hardy type theorems for the commutators of Calder\'on-Zygmund operators, we need the following three technical lemmas.

\begin{Lemma}\label{almost diagonal 1}
Let $\delta\in (0,1]$, and $A, B$ be two  $\delta$-Calder\'on-Zygmund operators such that $A 1=A^*1=B 1=B^*1=0$. Then, there exists a constant $C=C(n,\delta)$ such that
$$\sum_{I,I',I''}\sum_{\sigma,\sigma',\sigma''\in E}|\langle f,\psi^{\sigma}_I\rangle \langle g,\psi^{\sigma'}_{I'}\rangle \langle A\psi^{\sigma}_I, \psi^{\sigma''}_{I''}\rangle \langle B\psi^{\sigma'}_{I'},\psi^{\sigma''}_{I''}\rangle|\leq C \|f\|_{H^1}\|g\|_{BMO}$$
for all $f\in H^1(\mathbb R^n), g\in BMO(\mathbb R^n)$.
\end{Lemma}

\begin{Lemma}\label{almost diagonal 2}
Let $\delta\in (0,1]$, and $A_i, B_i$, $i=1,..., K$, be $\delta$-Calder\'on-Zygmund operators satisfying $A_i 1=A_i^* 1=B_i 1=B_i^* 1=0$, and for every $f$ and $g$ in $L^2(\mathbb R^n)$, 
$$\int_{\mathbb R^n}\Big(\sum_{i=1}^K A_i f .B_i g\Big)dx=0.$$
Then, the bilinear operator $\mathfrak P$, defined by $\mathfrak P(f,g)=\sum_{i=1}^K \mathfrak S(A_i f, B_i g)$, maps continuously   $H^1(\mathbb R^n)\times BMO(\mathbb R^n)$ into  $H^1(\mathbb R^n)$.
\end{Lemma}

\begin{Corollary}\label{operator S}
Let $T$ be a Calder\'on-Zygmund operator satisfying $T1=T^*1=0$. Then the bilinear operator $\mathfrak P$, defined by $\mathfrak P(f,g)= \mathfrak S(Tf, g)- \mathfrak S(f, T^*g)$, maps continuously  $H^1(\mathbb R^n)\times BMO(\mathbb R^n)$ into  $H^1(\mathbb R^n)$.
\end{Corollary}

\begin{Lemma}\label{technique}
Let $b$ be a non-constant $BMO$-function and $T$ be a Calder\'on-Zygmund operator with $T1=T^*1=0$. Assume that $f\in H^1_b(\mathbb R^n)$ has the wavelet decomposition $f=\sum_{j=1}^\infty \sum_{I\subset R_j}\sum_{\sigma\in E} \langle f,\psi_I^\sigma\rangle \psi_I^\sigma$ where the $R_j$'s are dyadic cubes and $\sum_{I\subset R_j}\sum_{\sigma\in E} \langle f,\psi_I^\sigma\rangle \psi_I^\sigma$ are multiples of $\psi$-atoms related to the cubes $R_j$. Set $f_k=\sum_{j=1}^k \sum_{I\subset R_j}\sum_{\sigma\in E} \langle f,\psi_I^\sigma\rangle \psi_I^\sigma$, $k=1,2,...$ Then, the sequence $\{[b,T](f_k)\}_{k\geq 1}$ tends to $[b,T](f)$ in the sense of distributions $\mathcal S'(\mathbb R^n)$.
\end{Lemma}

\begin{proof}[Proof of Lemma \ref{almost diagonal 1}]
We first remark (see \cite{MC}, Proposition 1) that  there exists a constant $C>0$ such that for all dyadic cubes $I, I'$ and $\sigma,\sigma'\in E$, we have
\begin{equation}\label{Meyer}
\max\{|\langle A\psi^{\sigma}_I, \psi^{\sigma'}_{I'}\rangle|, |\langle B\psi^{\sigma}_{I},\psi^{\sigma'}_{I'}\rangle|\}\leq C 2^{-|j-j'|(\delta+n/2)}\Big(\frac{2^{-j}+2^{-j'}}{2^{-j}+2^{-j'}+|x_I-x_{I'}|}\Big)^{n+\delta}.
\end{equation}
Consequently,
\begin{equation}\label{Dob}
\max\{|\langle A\psi^{\sigma}_I, \psi^{\sigma'}_{I'}\rangle|, |\langle B\psi^{\sigma}_{I},\psi^{\sigma'}_{I'}\rangle|\}\leq C p_\delta(I,I')
\end{equation}
with 
$$p_\delta(I,I')=\frac{2^{-|j-j'|(\delta/2+n/2)}}{1+ |j-j'|^2} \Big(\frac{2^{-j}+2^{-j'}}{2^{-j}+2^{-j'}+|x_I-x_{I'}|}\Big)^{n+\delta/2}.$$
Here  $|I|=2^{-jn}$ and $|I'|=2^{-j'n}$, while $x_I$ and $x_{I'}$ denote the centers of the two cubes. On the other hand, it follows from Lemma 1.3 in \cite{Do} that there exists a constant $C=C(n,\delta)>0$ such that
\begin{equation}\label{Dob2}
\sum_{I''}p_\delta(I,I'')p_\delta(I',I'')\leq C p_\delta(I,I').
\end{equation}
Combining (\ref{Dob}) and (\ref{Dob2}), we obtain 
$$\sum_{I,I',I''}\sum_{\sigma,\sigma',\sigma''\in E}|\langle f,\psi^{\sigma}_I\rangle \langle g,\psi^{\sigma'}_{I'}\rangle \langle A\psi^{\sigma}_I, \psi^{\sigma''}_{I''}\rangle \langle B\psi^{\sigma'}_{I'},\psi^{\sigma''}_{I''}\rangle|\leq C \sum_{I,I'}\sum_{\sigma,\sigma'\in E} p_\delta(I,I')|\langle f,\psi^{\sigma}_I\rangle|| \langle g,\psi^{\sigma'}_{I'}\rangle|.$$

It is easy to establish that the matrix $\{p_{\delta}(I, I')\}_{I,I'}$  is almost diagonal (by taking $\varepsilon= \delta/4$ in the definition (3.1) of Frazier and Jawerth \cite{FJ}) and thus is bounded on $\dot f^{0,2}_1$ the space  of all sequences $(a_I)_{I}$ such that $\Big(\sum_I |a_I|^2 |I|^{-1}\chi_I\Big)^{1/2}$ is in $L^1(\mathbb R^n)$. We then use the wavelet characterization of $H^1(\mathbb R^n)$ (see Theorem \ref{wavelet characterization}) and the fact that (cf. \cite{FJ})
$$\sum_{I'}\sum_{\sigma'\in E}|\langle h,\psi^{\sigma'}_{I'}\rangle| |\langle g,\psi^{\sigma'}_{I'}\rangle|\leq C \|h\|_{H^1}\|g\|_{BMO},$$
for all $h\in H^1(\mathbb R^n)$, to conclude that
$$\sum_{I,I',I''}\sum_{\sigma,\sigma',\sigma''\in E}|\langle f,\psi^{\sigma}_I\rangle \langle g,\psi^{\sigma'}_{I'}\rangle \langle A\psi^{\sigma}_I, \psi^{\sigma''}_{I''}\rangle \langle B\psi^{\sigma'}_{I'},\psi^{\sigma''}_{I''}\rangle|\leq C \|f\|_{H^1}\|g\|_{BMO}.$$

\end{proof}

\begin{proof}[Proof of Lemma \ref{almost diagonal 2}]
By Lemma \ref{almost diagonal 1}, we have
\begin{eqnarray*}
\mathfrak P(f,g)&=&\sum_{i=1}^K \mathfrak S(A_i f, B_i g)\\
&=&
\sum_{i=1}^K \sum_{I,I',I''}\sum_{\sigma,\sigma', \sigma''\in E}\langle f,\psi^\sigma_I\rangle \langle g,\psi^{\sigma'}_{I'}\rangle \langle A_i \psi^\sigma_I, \psi^{\sigma''}_{I''}\rangle \langle B_i \psi^{\sigma'}_{I'}, \psi^{\sigma''}_{I''}\rangle (\psi^{\sigma''}_{I''})^2
\end{eqnarray*}
where all the series converge in $L^1(\mathbb R^n)$. For any dyadic cubes $I,I'$, $\sigma,\sigma'\in E$, we have
\begin{eqnarray*}
&&\sum_{i=1}^K \sum_{I''}\sum_{\sigma''\in E}\langle f,\psi^\sigma_I\rangle \langle g,\psi^{\sigma'}_{I'}\rangle \langle A_i \psi^\sigma_I, \psi^{\sigma''}_{I''}\rangle \langle B_i \psi^{\sigma'}_{I'}, \psi^{\sigma''}_{I''}\rangle (\psi^{\sigma''}_{I''})^2\\
&&\quad= \sum_{i=1}^K \sum_{I''}\sum_{\sigma''\in E}\langle f,\psi^\sigma_I\rangle \langle g,\psi^{\sigma'}_{I'}\rangle \langle A_i \psi^\sigma_I, \psi^{\sigma''}_{I''}\rangle \langle B_i \psi^{\sigma'}_{I'}, \psi^{\sigma''}_{I''}\rangle \Big((\psi^{\sigma''}_{I''})^2- (\psi^\sigma_I)^2\Big)
\end{eqnarray*}
since (see  Remark \ref{remark})
$$\sum_{i=1}^K \sum_{I''}\sum_{\sigma''\in E} \langle A_i \psi^\sigma_I, \psi^{\sigma''}_{I''}\rangle \langle B_i \psi^{\sigma'}_{I'}, \psi^{\sigma''}_{I''}\rangle = \int_{\mathbb R^n}\Big(\sum_{i=1}^K A_i \psi^\sigma_I.  B_i \psi^{\sigma'}_{I'}\Big)dx=0.$$
 An explicit computation gives that $|\psi_{I''}^{\sigma''}|^2-|\psi_{I}^{\sigma}|^2$ is in $H^1(\bR^n)$, with
$$\||\psi_{I''}^{\sigma''}|^2-|\psi_{I}^{\sigma}|^2\|_{H^1}\leq C\left(\log (2^{-j}+2^{-j''})^{-1}+\log(|x_I-x_{I''}|+2^{-j}+2^{-j''})\right).$$
Here  $|I|=2^{-jn}$ and $|I''|=2^{-j''n}$, while $x_I$ and $x_{I''}$ denote the centers of the two cubes. Consequently, by (\ref{Meyer}) and (\ref{Dob2}), we get
\begin{eqnarray*}
&& \Big\|\sum_{i=1}^K \sum_{I''}\sum_{\sigma''\in E}\langle f,\psi^\sigma_I\rangle \langle g,\psi^{\sigma'}_{I'}\rangle \langle A_i \psi^\sigma_I, \psi^{\sigma''}_{I''}\rangle \langle B_i \psi^{\sigma'}_{I'}, \psi^{\sigma''}_{I''}\rangle (\psi^{\sigma''}_{I''})^2\Big\|_{H^1}\\
&\leq&
\sum_{i=1}^K \sum_{I''}\sum_{\sigma''\in E}|\langle f,\psi^\sigma_I\rangle \langle g,\psi^{\sigma'}_{I'}\rangle \langle A_i \psi^\sigma_I, \psi^{\sigma''}_{I''}\rangle \langle B_i \psi^{\sigma'}_{I'}, \psi^{\sigma''}_{I''}\rangle|\Big\|(\psi^{\sigma''}_{I''})^2- (\psi^\sigma_I)^2\Big\|_{H^1}\\
&\leq&
C \sum_{i=1}^K \sum_{I''}\sum_{\sigma''\in E}|\langle f,\psi^\sigma_I\rangle \langle g,\psi^{\sigma'}_{I'}\rangle| p_\delta(I,I'')p_\delta(I',I'')\\
&\leq&
C  p_\delta(I,I')|\langle f,\psi^\sigma_I\rangle| |\langle g,\psi^{\sigma'}_{I'}\rangle|,
\end{eqnarray*}
here we used the fact that
$$(1+ |j-j''|^2)\log\Big(\frac{|x_I-x_{I''}|+2^{-j}+2^{-j''}}{2^{-j}+2^{-j''}}\Big)\leq C(\delta) 2^{|j-j''|\delta/2}\Big(\frac{|x_I-x_{I''}|+2^{-j}+2^{-j''}}{2^{-j}+2^{-j''}}\Big)^{\delta/2}.$$
Thus, the same argument as in the proof of Lemma \ref{almost diagonal 1} allows to conclude that
\begin{eqnarray*}
\|\mathfrak P(f,g)\|_{H^1}&\leq& C \sum_{I,I'}\sum_{\sigma,\sigma'\in E} p_\delta(I,I')|\langle f,\psi^\sigma_I\rangle| |\langle g,\psi^{\sigma'}_{I'}\rangle| \\
&\leq&
C \|f\|_{H^1}\|g\|_{BMO},
\end{eqnarray*}
which ends the proof.

\end{proof}

Before giving the proof of  Lemma \ref{technique}, let us recall the following lemma. It can be found in \cite{FTW}.\\\\
{\bf Lemma A.} (see \cite{FTW}, Lemma 2.3)
Let $T$ be a Calder\'on-Zygmund operator satisfying $T1=0$. Then $T$ maps $\mathcal S(\mathbb R^n)$ into $L^\infty(\mathbb R^n)$. Moreover, there exists a constant $C>0$, depending only on $T$, such that for any $\phi\in \mathcal S(\mathbb R^n)$ with supp $\phi\subset B(x_0,r)$, we have
$$\|T\phi\|_{L^\infty}\leq C (\|\phi\|_{L^\infty}+ r\||\nabla \phi|\|_{L^\infty}).$$

\begin{proof}[Proof of Lemma \ref{technique}]
By Theorem \ref{Th1}, it is sufficient to prove that 
$$\lim_{k\to\infty} \int_{\mathbb R^n} T(\mathfrak S(f_k, b))h dx= \int_{\mathbb R^n} T(\mathfrak S(f, b))h dx,$$
for all $h\in \mathcal S(\mathbb R^n)$. Because of the hypothesis, we observe that  $\mathfrak S(f, b)\in H^1(\mathbb R^n)$ and $\mathfrak S(f_k, b)\in L^q(\mathbb R^n)$, $k=1, 2,...$, for some $q\in (1,2)$ (see Lemma \ref{Nirenberg}).

 Let $\mathfrak S(f, b)= \sum_{j=1}^\infty \lambda_j a_j$ be a classical $L^q$-atomic decomposition of $\mathfrak S(f, b)$. Then, $T(\sum_{j=1}^k \lambda_j a_j)$ tends to $T(\mathfrak S(f, b))$ in $L^1(\mathbb R^n)$ (in fact, it also holds in $H^1(\mathbb R^n)$ since $T^*1=0$). Hence, as $h\in \mathcal S(\mathbb R^n)\subset L^\infty(\mathbb R^n)\cap L^{q'}(\mathbb R^n)$ where $1/q+1/{q'}=1$, $ \mathfrak S(f_k, b), a_j\in L^q(\mathbb R^n)$ and $T^*h\in L^\infty(\mathbb R^n)$ since $T^*1=0$ (see Lemma A), by Theorem A  we get
\begin{eqnarray*}
\int_{\mathbb R^n} T(\mathfrak S(f, b)) hdx&=&\lim_{k\to\infty} \int_{\mathbb R^n} T\Big(\sum_{j=1}^k \lambda_j a_j\Big) hdx =  \lim_{k\to\infty} \int_{\mathbb R^n} \Big(\sum_{j=1}^k \lambda_j a_j\Big) T^*h dx\\
&=& \int_{\mathbb R^n}\mathfrak S(f, b) T^*h dx= \lim_{k\to\infty} \int_{\mathbb R^n} \mathfrak S(f_k, b)T^*h dx\\
&=& \lim_{k\to\infty} \int_{\mathbb R^n} T(\mathfrak S(f_k, b))h dx,
\end{eqnarray*}
since $\mathfrak S(f_k, b)$ tends to $\mathfrak S(f, b)$ in $L^1(\mathbb R^n)$ as $f_k$ tends to $f$ in $H^1(\mathbb R^n)$ (see Theorem \ref{Th1}). This finishes the proof.

\end{proof}

\begin{proof}[Proof of Theorem \ref{bilinear estimates}]
Let $(f,g)\in H^1(\mathbb R^n)\times BMO(\mathbb R^n)$. By Theorem \ref{Th1} and Lemma \ref{almost diagonal 2}, we obtain $\mathfrak T(f,g)= \sum_{i=1}^K [B_i g, T](A_i f)\in L^1(\mathbb R^n)$, moreover,
\begin{eqnarray*}
\|\mathfrak T(f,g)\|_{L^1}&\leq& \sum_{i=1}^K \|\mathfrak R(A_i f, B_i g)\|_{L^1}+ \Big\|T\Big(\sum_{i=1}^K \mathfrak S(A_i f, B_i g)\Big)\Big\|_{L^1}\\
&\leq&
C \sum_{i=1}^K \|A_i f\|_{H^1}\|B_i g\|_{BMO}+ \|T\|_{H^1\to L^1} \Big\|\sum_{i=1}^K \mathfrak S(A_i f, B_i g)\Big\|_{H^1}\\
&\leq&
C \|f\|_{H^1}\|g\|_{BMO}.
\end{eqnarray*}
This completes the proof.

\end{proof}

\begin{proof}[Proof of Theorem \ref{Th3}]
Let $f\in H^1_b(\mathbb R^n)$, we prove $[b,T](f)\in h^1(\mathbb R^n)$ using the fact that $BMO^{\rm log}(\mathbb R^n)$ is the dual of $H^{\rm log}(\mathbb R^n)$ (see \cite{Ky}). Indeed, by Theorem \ref{atomic decomposition}, there exists a decomposition $f= \sum_{j=1}^\infty \sum_{I\subset R_j}\sum_{\sigma\in E} \langle f,\psi_I^\sigma\rangle \psi_I^\sigma$ where $\sum_{I\subset R_j}\sum_{\sigma\in E} \langle f,\psi_I^\sigma\rangle \psi_I^\sigma$ are multiples of $\psi$-atoms related to the dyadic cubes $R_j$. Set $f_k= \sum_{j=1}^k \sum_{I\subset R_j}\sum_{\sigma\in E} \langle f,\psi_I^\sigma\rangle \psi_I^\sigma$, $k=1,2,...$ Then, the sequence $[b,T](f_k)$ tends to $[b,T](f)$ in the sense of distributions $\mathcal S'(\mathbb R^n)$ (see Lemma \ref{technique}), and thus
\begin{equation}\label{lim}
\lim\limits_{k\to\infty}\int_{\mathbb R^n} [b,T](f_k) hdx = \int_{\mathbb R^n} [b,T](f) hdx,
\end{equation}
for all $h\in C^\infty_0(\mathbb R^n)$. Notice that $[b,T](f_k)\in L^2(\mathbb R^n)$ and $[b,T](f)\in L^1(\mathbb R^n)$.

Let $h\in C^\infty_0(\mathbb R^n)$.  By Lemma \ref{Le2}, Lemma \ref{Le3},  Lemma \ref{Le4}, Remark \ref{remark} and Corollary \ref{operator S}, we have $hT(f_k)- f_k\Big(T^*h- (T^*h)_{\mathbb Q}\Big)\in H^{\rm log}(\mathbb R^n)$. More precisely,
\begin{eqnarray*}
&&\Big\|hT(f_k)- f_k\Big(T^*h- (T^*h)_{\mathbb Q}\Big) \Big\|_{H^{\rm log}}\\
&\leq& C\Bigg\{ \Big\|\mathfrak S(T(f_k), h)- \mathfrak S\Big(f_k, T^*h-(T^*h)_{\mathbb Q}\Big)\Big\|_{H^1}+\\
&&
\quad\quad+ \sum_{j=1,4}\left(\|\Pi_j(T(f_k), h)\|_{H^1}+ \Big\|\Pi_j\Big(f_k, T^*h-(T^*h)_{\mathbb Q}\Big)\Big\|_{H^1}\right)+\\
&&
\quad\quad\quad\quad\quad\quad+ \|\Pi_2(T(f_k), h)\|_{H^{\rm log}} + \Big\|\Pi_2\Big(f_k, T^*h-(T^*h)_{\mathbb Q}\Big)\Big\|_{H^{\rm log}}\Bigg\}\\
&\leq&
C \Bigg\{\|f_k\|_{H^1}\|h\|_{BMO}+ \|T(f_k)\|_{H^1}\|h\|_{BMO}+ \|f_k\|_{H^1}\Big\|T^*h-(T^*h)_{\mathbb Q}\Big\|_{BMO} +\\
 &&
\quad\quad\quad\quad\quad\quad\quad\quad\quad + \|T(f_k)\|_{H^1}\|h\|_{BMO^+}+ \|f_k\|_{H^1}\|T^*h-(T^*h)_{\mathbb Q}\|_{BMO^+} \Bigg\}\\
&\leq&
C (\|f_k\|_{H^1} \|h\|_{bmo}+ \|f_k\|_{H^1} \|T^*h\|_{BMO})\leq C \|f\|_{H^1}\|h\|_{bmo},
\end{eqnarray*}
here one used $\mathfrak S\Big(f, T^*h-(T^*h)_{\mathbb Q}\Big)= \mathfrak S(f, T^*h)$, $\|T^*h-(T^*h)_{\mathbb Q}\|_{BMO^+}=\|T^*h\|_{BMO}$ and $\|f_k\|_{H^1}\leq C \|f\|_{H^1}$. As the $L^2$- functions $f_k$ have compact support, $b\in BMO^{\rm log}(\mathbb R^n)\subset BMO(\mathbb R^n)$, we deduce that $bhT(f_k), hT(bf_k), bf_kT^*h\in L^1(\mathbb R^n)$. Moreover, $\int_{\mathbb R^n} hT(bf_k)dx= \int_{\mathbb R^n} bf_kT^*hdx$ since $ hT(bf_k)- bf_kT^*h\in H^1(\mathbb R^n)$  (see Theorem A). Therefore, as $BMO^{\rm log}(\mathbb R^n)$ is the dual of $H^{\rm log}(\mathbb R^n)$ (see \cite{Ky}), we get 
\begin{eqnarray*}
\left|\int_{\mathbb R^n} [b,T](f_k) h dx \right|&=& \left|\int_{\mathbb R^n} b (hT(f_k)- f_kT^*h)dx\right|\\
&\leq& \left|\int_{\mathbb R^n} b \left(hT(f_k)- f_k\Big(T^*h- (T^*h)_{\mathbb Q}\Big)\right)dx\right| + |(T^*h)_{\mathbb Q}|\left|\int_{\mathbb R^n} bf_k dx\right|\\
&\leq&
C \|b\|_{BMO^{\rm log}}\Big\|hT(f_k)- f_k\Big(T^*h- (T^*h)_{\mathbb Q}\Big) \Big\|_{H^{\rm log}} + |(T^*h)_{\mathbb Q}|\left|\int_{\mathbb R^n} bf_k dx\right|\\
&\leq&
C \|b\|_{BMO^{\rm log}} \|f\|_{H^1}\|h\|_{bmo} + |(T^*h)_{\mathbb Q}| \Big|\sum_{j=1}^k \sum_{I\subset R_j}\sum_{\sigma\in E}\langle f,\psi_I^\sigma\rangle \langle b,\psi_I^\sigma\rangle\Big|.
\end{eqnarray*}

The above inequality and (\ref{lim}) imply that for all $h\in C^\infty_0(\mathbb R^n)$, we obtain
$$\left|\int_{\mathbb R^n} [b,T](f) h dx \right|\leq C \|b\|_{BMO^{\rm log}} \|f\|_{H^1}\|h\|_{bmo}$$
since $\mathfrak S(f,b)\in H^1(\mathbb R^n)$ (see Theorem \ref{Bonami}) and thus (see Remark \ref{remark})
$$\lim_{k\to \infty} \sum_{j=1}^k \sum_{I\subset R_j}\sum_{\sigma\in E}\langle f,\psi_I^\sigma\rangle \langle b,\psi_I^\sigma\rangle=\int_{\mathbb R^n} \mathfrak S(f,b)dx=0.$$
This proves that $[b,T](f)\in h^1(\mathbb R^n)$ since $h^1(\mathbb R^n)$ is the dual of $vmo(\mathbb R^n)$ (see Section 2). Furthermore, 
$$\|[b,T](f)\|_{h^1}\leq C \|b\|_{BMO^{\rm log}} \|f\|_{H^1}\leq C \|b\|_{BMO^{\rm log}}\|b\|^{-1}_{BMO} \|f\|_{H^1_b},$$
which ends the proof of Theorem \ref{Th3}.

\end{proof}

\begin{proof}[Proof of Theorem \ref{Th4}]
By Theorem \ref{Th1} and Theorem \ref{Bonami} together with Lemma \ref{Le2} and Lemma \ref{Le3}, it is sufficient to prove that the linear operator
$$f\mapsto \mathfrak U(f,b):= bTf- T(\Pi_2(f,b))$$
is bounded from $H^1(\mathbb R^n)$ into itself. Similarly to the proof of Theorem \ref{new1}, we first consider $f$ a  $\psi$-atom related to the cube $Q = Q[x_0, r]$ and note that
\begin{equation}\label{lost 1}
\mathfrak U(f,b)= \mathfrak U(f,b- b_Q)= (b-b_Q) Tf- T(\Pi_2(f,b-b_Q)).
\end{equation}

Let  $\varepsilon\in (0, 1)$, recall that (see \cite{TW}) $g$ is an $\varepsilon$-molecule for $H^1(\mathbb R^n)$ centered at $y_0$ if
$$\int_{\mathbb R^n}g(x)dx=0\quad\mbox{and}\quad \|g\|_{L^q}^{1/2}\|g|\cdot-y_0|^{2n\varepsilon}\|_{L^q}^{1/2}=: \mathfrak N(g)<\infty,$$
where $q=1/(1-\varepsilon)$. It is well known that if $g$ is an $\varepsilon$-molecule for $H^1(\mathbb R^n)$ centered at $y_0$, then $g\in  H^1(\mathbb R^n)$ and $\|g\|_{H^1}\leq C \mathfrak N(g)$ where $C > 0$ depends only on $n, \varepsilon$.

We now prove that $(b - b_Q)Tf$ is an $\varepsilon$-molecule for $H^1(\mathbb R^n)$ centered at $x_0$ when $T$ is a $\delta$-Calder\'on-Zygmund operator for some $\delta\in (0, 1]$ and $\varepsilon =\delta/(4n) < 1/2$. Note first that $f$ is $C$ times a classical $L^2$-atom related to the cube $mQ$. It is clear that $\int_{\mathbb R^n}(b-b_Q)Tf dx=0$ since $T^*1=T^*b=0$. As $q=1/(1-\varepsilon)<2$, the fact $|b_Q- b_{2mQ}|\leq C \|b\|_{BMO}$ together with H\"older inequality and John-Nirenberg inequality, give
\begin{equation}\label{lost 2}
\|(b - b_Q)Tf. \chi_{2mQ}\|_{L^q}\leq C |Q|^{1/q-1}\|b\|_{BMO}.
\end{equation}
It is well-known that $|Tf(x)|\leq C\frac{r^\delta}{|x-x_0|^{n+\delta}}$, for all $x\in (2mQ)^c$, since $T$ is a $\delta$-Calder\'on-Zygmund operator. Hence
\begin{eqnarray*}
\|(b - b_Q)Tf. \chi_{(2mQ)^c}\|_{L^q}&\leq& C \left(\int_{(2mQ)^c} |b-b_Q|^q\Big(\frac{r^\delta}{|x-x_0|^{n+\delta}}\Big)^{q}dx\right)^{1/q}\\
&\leq& C |Q|^{1/q-1}\|b\|_{BMO}.
\end{eqnarray*}
The last inequality, which can be found in \cite{St}, is classical. Combining this and (\ref{lost 2}), we obtain
\begin{equation}\label{lost 3}
\|(b - b_Q)Tf\|_{L^q}\leq C |Q|^{1/q-1}\|b\|_{BMO}.
\end{equation}

Similarly, we also have
$$\|(b - b_Q)Tf. |\cdot-x_0|^{2n\varepsilon}.\chi_{2mQ}\|_{L^q}\leq C |Q|^{2\varepsilon+ 1/q-1}\|b\|_{BMO}$$
and as $2n\varepsilon= \delta/2$,
\begin{eqnarray*}
\|(b - b_Q)Tf. |\cdot-x_0|^{2n\varepsilon}. \chi_{(2mQ)^c}\|_{L^q}&\leq& C \left(\int_{(2mQ)^c} |b-b_Q|^q\Big(\frac{r^\delta}{|x-x_0|^{n+\delta/2}}\Big)^{q}dx\right)^{1/q}\\
&\leq& C |Q|^{2\varepsilon+ 1/q-1}\|b\|_{BMO}.
\end{eqnarray*}
Consequently,
$$\|(b - b_Q)Tf. |\cdot-x_0|^{2n\varepsilon}\|_{L^q}\leq C |Q|^{2\varepsilon+ 1/q-1}\|b\|_{BMO}.$$
Combining this and (\ref{lost 3}), we get $(b- b_Q)Tf$ is an $\varepsilon$-molecule for $H^1(\mathbb R^n)$ centered
at $x_0$, moreover,
$$\mathfrak N((b- b_Q)Tf)\leq C |Q|^{\varepsilon+ 1/q-1}\|b\|_{BMO}\leq C \|b\|_{BMO},$$
since $q=1/(1-\varepsilon)$. Thus, by (\ref{lost 1}) and Remark \ref{remark},
\begin{equation}\label{lost 4}
\|\mathfrak U(f,b)\|_{H^1}\leq C \mathfrak N((b- b_Q)Tf)+ \|T(\Pi_2(f,b-b_Q))\|_{H^1}\leq C \|b\|_{BMO}.
\end{equation}

Now, let us consider $f\in H^1_{\rm fin}(\mathbb R^n)$. By Lemma \ref{finite decomposition}, there exists a finite decomposition $f=\sum_{j=1}^k \lambda_j a_j$ such that $\sum_{j=1}^k |\lambda_j|\leq C \|f\|_{H^1}$. Consequently, by (\ref{lost 4}),
we obtain that
$$\|\mathfrak U(f,b)\|_{H^1}\leq \sum_{j=1}^k |\lambda_j| \|\mathfrak U(a_j,b)\|_{H^1}\leq C \|f\|_{H^1}\|b\|_{BMO},$$
which ends the proof as $H^1_{\rm fin}(\mathbb R^n)$ is dense in $H^1(\mathbb R^n)$ for the norm $\|\cdot\|_{H^1}$.

\end{proof}

\section{Commutators of Fractional integrals}

 Given $0<\alpha< n$, the fractional integral operator $I_\alpha$ is defined by
$$I_\alpha f(x)= \int_{\mathbb R^n} \frac{f(y)}{|x-y|^{n-\alpha}}dy.$$
Let $b$ be a locally integrable function. We consider the linear commutator $[b, I_\alpha]$ defined by
$$[b, I_\alpha](f)= b I_\alpha f- I_\alpha (bf).$$

We end this article by presenting some results related to commutators of fractional integrals as follows.

\begin{Theorem}
Let $0<\alpha< n$. There exist a bounded bilinear operator $\mathfrak R: H^1(\mathbb R^n)\times BMO(\mathbb R^n)\to L^{n/(n-\alpha)}(\mathbb R^n)$ and a bounded bilinear operator $\mathfrak S: H^1(\mathbb R^n)\times BMO(\mathbb R^n)\to L^1(\mathbb R^n)$ such that
$$[b, I_\alpha](f)= \mathfrak R(f,b)+ I_\alpha(\mathfrak S(f,b)).$$
\end{Theorem}

\begin{Corollary}
Let $0<\alpha< n$ and $b\in BMO(\mathbb R^n)$. Then, the linear commutator $[b, I_\alpha]$ maps continuously $H^1(\mathbb R^n)$ into weak-$L^{n/(n-\alpha)}(\mathbb R^n)$.
\end{Corollary}

\begin{Theorem}\label{the last theorem}
Let $0<\alpha< n$, $b\in BMO(\mathbb R^n)$, and $1<q\leq \infty$. Then, the linear commutator $[b, I_\alpha]$ maps continuously $ H^{1}_b(\mathbb R^n)$ into $L^{n/(n-\alpha)}(\mathbb R^n)$.
\end{Theorem}

The results above can be proved similarly to  Theorem \ref{Th1} and Theorem \ref{new2}. We leave the proofs to the interested readers. When $H^{1}_b(\mathbb R^n)$ is replaced by $\mathcal  H^{1}_b(\mathbb R^n)$, Theorem \ref{the last theorem} was considered by the authors in \cite{DLZ}. There, they  proved that
$$\sup\{\|[b,I_\alpha](a)\|_{L^{n/(n-\alpha)}}: a \,\mbox{is a}\; (\infty,b){\rm -atom}\}<\infty.$$
However, as pointed out before, this argument does not suffice to conclude that $[b,I_\alpha]$ is bounded from $\mathcal H^{1}_b(\mathbb R^n)$ into $L^{n/(n-\alpha)}(\mathbb R^n)$.

\end{document}